\documentclass[12pt]{amsart}

\usepackage {amsmath, amssymb, amscd, graphicx, color, url,epsf}
\usepackage{amsmath, amssymb, graphicx, epsfig, verbatim, graphs, psfrag,enumitem}

\newcounter{bean}
\Alph{bean}
\newtheorem{thm}{Theorem}[section]

\newtheorem{cor}[thm]{Corollary}
\newtheorem{lem}[thm]{Lemma}
\newtheorem{prop}[thm]{Proposition}
\newtheorem{defn}[thm]{Definition}

\newtheorem{rem}[thm]{Remark}

\numberwithin{equation}{section}

\newcommand{\x}{\mathbf x}

\newcommand{\y}{\mathbf y}
\newcommand\Ta{\mathbb{T}_{\alpha}}
\newcommand\Tb{\mathbb{T}_{\beta}}

\newcommand\CFm{\mathrm{CF}^-}
\newcommand\CFmp{{{\widetilde{\mathrm{CF}}}}{}^-}

\newcommand\HFa{\widehat{\mathrm{HF}}}
\newcommand\HFm{\mathrm{HF}^-}
\newcommand\HFn{\mathrm{HF}^-_{[n]}}

\newcommand\HFtwo{\mathrm{HF}^-_{[2]}}
\newcommand\CFtwo{\mathrm{CF}^-_{[2]}}

\newcommand{\smargin}[1]{\marginpar{\tiny{#1}}}
\newcommand{\Field}{\mathbb{F}}
\newcommand{\SpinC}{\mathrm{Spin}^c}
\newcommand{\adapted}{adapted }

\newcommand{\Z}{\mathbb Z}

\newcommand{\R}{\mathbb R}

\newcommand{\bfr}{\mathbb R}

\newcommand\Mas{\mu}

\DeclareMathOperator{\Sym}{Sym}
\newcommand{\alphas}{{\mathbf {\alpha}}}
\newcommand{\betas}{{\mathbf {\beta}}}
\newcommand\as{\mathbf a}
\newcommand\bs{\mathbf b}
\newcommand\ws{\mathbf w}
\newcommand\Os{\mathbf O}
\newcommand\Xs{\mathbf X}

\DeclareMathOperator{\ModFlow}{\mathcal M}
\DeclareMathOperator{\UnparModFlow}{\widehat\ModFlow}
\begin{document}

\title[$U^2=0$]{A combinatorial description of the $U^2=0$ version of Heegaard Floer homology}

\author{Peter S. Ozsv\'ath}
\address{Department of Mathematics\\
Columbia University, New York, NY, 10027}
\email{petero@math.columbia.edu}

\author{Andr\'{a}s I. Stipsicz}
\address{Department of Mathematics\\
Columbia University, New York, NY, 10027  and \\
R{\'e}nyi Institute of Mathematics\\
Budapest, Hungary}
\email{stipsicz@math-inst.hu}

\author{Zolt\'an Szab\'o}
\address{Department of Mathematics\\
Princeton University\\
Princeton, NJ, 08450}
\email{szabo@math.princeton.edu}

\begin{abstract}
We show that every 3--manifold admits a Heegaard diagram in which 
a truncated version of Heegaard Floer homology (when the holomorpic
disks pass through the basepoints at most once) can be computed
combinatorially.
\end{abstract}
\maketitle

\section{Introduction}
\label{sec:first}

Heegaard Floer homology~\cite{OSz1,OSz2} is a topological invariant
for closed, oriented three-manifolds.  These invariants are defined by
considering versions of Lagrangian Floer homologies of certain tori
(derived from Heegaard diagrams of $Y$) in a symmetric power of the
Heegaard surface.  There are several variants of these homology
groups: there is a simple version $\HFa (Y)$, where only those
holomorphic disks are counted which are supported in the complement of
a certain divisor in the symmetric power, and a richer version $\HFm
(Y)$, where the holomorphic disks passing through the above mentioned
divisor are counted with weight, recorded in the exponent of a formal
variable $U$. Specifically, this latter group admits the algebraic
structure of an $\Field [U]$--module, where $\Field=\Z/2\Z$, and the
corresponding $\HFa$--theory can be derived from it by setting $U=0$
at the chain complex level.

The computation of these invariants is typically challenging, since
one needs to analyze moduli spaces of pseudo--holomorphic disks in
high symmetric powers of the Heegaard surface. It was noticed by
Sarkar and Wang \cite{SW} that for convenient Heegaard diagrams the
complex analytic considerations are unnecessary for the purpose of
calculating $\HFa (Y)$. Specifically, the combinatorial structure of
such a convenient diagram determines the chain complex computing the
Heegaard Floer homology group $\HFa (Y)$. In addition, Sarkar and Wang
showed that any 3--manifold admits such a convenient Heegaard diagram
(wich they called \emph{nice}).  In a related vein,~\cite{MOS} have
used grid diagrams to calculate all the variants of Heegaard Floer
homology of knots and links in $S^3$ in a purely combinatorial manner.

In this work we extend the result of Sarkar and Wang, by showing that
for a suitably chosen Heegaard diagram the differential of the chain
complex which counts holomorphic disks passing through the basepoint
at most once (that is, the chain complex obtained by setting $U^2=0$
in the chain complex computing $\HFm (Y)$) can also be determined
combinatorially. Furthermore, we show that every closed, oriented
3--manifold $Y$ admits such a Heegaard diagram.  

Indeed, letting $\CFm(Y)$ denote the free $\Field[U]$--module whose
homology is the invariant $\HFm(Y)$, and defining $\HFn(Y) $ as the
homology of the chain complex $(\CFm (Y)/(U^n=0), \partial ^-)$, then
we prove the following:

\begin{thm}
\label{t:main}
A closed, oriented 3--manifold $Y$ admits a Heegaard diagram from
which $\HFtwo (Y)$, the specialization to $\Field[U]/U^2$, can be
computed in a purely combinatorial manner.
\end{thm}

A more precise version of our result is stated in
Theorem~\ref{t:contribution}, after we introduce the particular
Heegaard diagram which is of use to us. For this diagram, the theorem
then says that all non-negative domains of Whitney disks which have
Maslov index one and which cross the distinguished divisor no more
than once do, in fact, contribute with multiplicity one in the
differential for $\CFtwo(Y)$. Finding such domains is clearly a
combinatorial matter. Our diagram is nice in the sense of Sarkar and
Wang, hence for Whitney disks which are {\em disjoint} from the
divisor, the result stated above actually follows from their work. The
result is new for the case of disks which cross the distinguished
divisor with multiplicity exactly one.

The Heegaard diagram considered in this paper is obtained from a
triple branched cover of a link, represented by a grid diagram.  As
such, our constructions here were partly motivated by work of Levine
\cite{Lev}, who studied branched covers of grid diagrams in a slightly
different context.  The construction of the Heegaard diagrams equips
them with some extra structures, most notably with a projection to a
grid diagram.  This projection can be conveniently used in
understanding properties of Maslov index zero and one homology
classes.

The optimistic reader might hope that an analogous result --- that
non-negative, Maslov index one homology classes of Whitney disks for
our Heegaard diagram always contribute with multiplicity one in the
differential --- could hold for disks which cross the distinguished
divisor an arbitrary number of times.  This fails, however, for our
diagram already for the case where disks cross the basepoints with
multiplicity two, i.e. for the $U^3=0$ version of the
theory. Specifically, some Maslov index one, non-negative homology
classes contribute to the differential, while others do not, and this
choice is determined analytically. However, there is a way of
formalizing this analytic choice combinatorially, and hence to come up
with a concrete chain complex calculating $\HFn(Y)$ for $n=3$. We will
return to this problem in~\cite{u3}.

The paper is organized as follows.  In Section~\ref{sec:Preliminaries}
we recall the basic notions regarding Heegaard Floer theory, while in
Section~\ref{sec:HeegaardDiagrams} we describe the special type of
Heegaard diagrams we would like to work with.  In
Section~\ref{sec:fifth} we classify the Maslov index one domains
crossing the distinguished divisor at most once in the Heegaard
diagram (these are the ones we encounter in determining the
differential), and in Section~\ref{sec:CalcContributions} we determine
the contribution of each Maslov index one domain in the differential,
concluding the evaluation of $\HFtwo (Y)$.

{\bf Acknowledgements}: PSO was supported by NSF grant number
DMS-0505811 and FRG-0244663.  AS was supported by the Clay Mathematics
Institute, by OTKA T49449 and by Marie Curie TOK project
BudAlgGeo. ZSz was supported by NSF grant number DMS-0704053 and
FRG-0244663.

\section{Preliminaries}
\label{sec:Preliminaries}

\subsection*{Heegaard Floer homologies}
Suppose that the closed, oriented 3--manifold $Y$ is given by a
Heegaard diagram $(\Sigma , \alphas , \betas )$, where $\Sigma$ is a
surface with genus $g$, $\alphas = \{ \alpha _1 , \ldots , \alpha _g
\}$ and $\betas = \{ \beta _1, \ldots , \beta _g \} $ are two
collections of disjoint, linearly independent, simple, closed curves
on $\Sigma$ with the property that the $\alpha _i $'s intersect the
$\beta _j$'s transversely. The 3--manifold $Y$ can be reconstructed
from the triple $(\Sigma , \alphas , \betas )$ by attaching
3--dimensional 2--handles along $\alpha _i \times \{ -1 \}$ and $\beta
_j \times \{ 1\}$ to $\Sigma \times [-1,1]$ and then attaching two
3--balls to the resulting 3--manifold with two $S^2$--boundaries. The
$\alpha $-- (and similarly the $\beta$--) curves give rise to a torus
$\Ta$ (resp.  $\Tb$) in the $g^{th}$ symmetric power $\Sym  ^g
(\Sigma )$ of $\Sigma$. Define $\CFm (Y)$ as the free module over
$\Field[U]$ (with $\Field=\Z/2\Z$) generated by the (finitely many)
intersection points $\Ta \cap \Tb$. Fix a point $w$ on $\Sigma$ which
is in the complement of the $\alpha$-- and $\beta$--curves.  The
differential on $\CFm (Y)$ is defined on an element $\x \in \Ta \cap
\Tb$ of the generating system by
\begin{equation}
\label{eq:DefineDifferential}
\partial ^- \x =\sum _{\y \in \Ta\cap \Tb}\sum _{\phi \in \pi _2 (\x , \y ),
\mu (\phi )=1} \# \UnparModFlow(\phi ) \cdot U^{n_w(\phi )}\y,
\end{equation}
where $\pi_2(\x,\y)$ denotes the space of homology classes of Whitney
disks connecting $\x$ to $\y \in \Ta\cap\Tb$, $\mu (\phi )$ denotes
the Maslov index of the homology class $\phi$, $n_w (\phi )$ is the
algebraic intersection number of $\phi$ with the divisor $V_w=\{ w\}
\times \Sym ^{g-1}(\Sigma )\subset \Sym ^g (\Sigma )$, and
${\UnparModFlow}(\phi )$ is the quotient of the moduli space of
holomorphic disks representing $\phi$ by the $\bfr$--action induced by
translation, with respect to some suitably-chosen path of
almost-complex structure $J$ on the symmetric product.  We typically
suppress the choice of $J$ from the notation; though the reader is
cautioned that the actual differential typically does depend on this
choice.  Sometimes, when we wish to emphasize this choice, we include
it in the notation for the moduli space, writing
$\UnparModFlow_J(\phi)$ to denote the space of $J$-pseudo-holomorphic
representatives of $\phi$, modulo the $\bfr$--action. The term $\#
{\UnparModFlow}(\phi)$ denotes the number of elements modulo 2 in the
quotient space.  The theory admits a sign refined version, which is
defined over $\Z [U]$, but we will not address sign issues in this
paper.

We can decompose $\Sigma-\alphas-\betas$ as a disjoint union of
regions $\coprod_i D_i$ which we call {\em elementary domains}.

\begin{defn}
  \label{def:Domain}
{\rm A homology class of Whitney disks $\phi\in\pi_2(\x,\y)$ naturally
  determines a formal linear combination $D(\phi )$ of elementary
  domains by fixing a point $d_i$ in each $D_i$ and associating $\sum
  _i n _{d_i}(\phi ) D_i$ to $\phi $. The linear combination $D(\phi
  )$ is called the {\em domain associated to $\phi$}. We say that
  $\phi$ is \emph{non-negative} if all the coefficients $n_{d_i}(\phi
  )$ in the above sum are non-negative, and $\phi$ is \emph{embedded}
  if each of the coefficients $n_{d_i}(\phi)$ in the above sum are
  either zero or one.}
\end{defn}

The sums in \eqref{eq:DefineDifferential} defining the differential
are not {\em a priori} finite for an arbitrary diagram. Finiteness is
ensured, however, if we work with {\em admissible} diagrams (called
{\em weakly admissible} in~\cite{OSz1}) which we briefly recall here.
A {\em periodic domain} for a Heegaard diagram is a two-chain with
boundary among the $\alpha_i$ and $\beta_j$, and with local
multiplicity zero at $w$. The space of periodic domains is an Abelian
group, which is isomorphic to $H_2(Y; \Z)$.  A Heegaard diagram is {\em
admissible} if every non-trivial periodic domain has both positive and
negative local multiplicities. It turns out that this condition
suffices to show that there are only finitely many non-zero terms with
some fixed $U$-power appearing in $\partial^- \x$ in the sum of
\eqref{eq:DefineDifferential}. To see then that the differential gives a
well-defined polynomial in $U$ (rather than a formal power series), we would
require stronger notions of admissibility (this is {\em strong admissibility}
of~\cite{OSz1}; note that this notion depends also on a choice of background
$\SpinC$ structure).  However, if we consider a variant of the theory (as we
will do here) where we set $U^n=0$ for some fixed $n\geq 1$, then the weak
form of admissibility suffices to ensure that the sums are finite, so
$(\CFm(Y)/(U^n=0), \partial ^-)$ is a chain complex.

In \cite{OSz1} it is shown that the homology of the resulting chain
complex is an invariant of the 3--manifold $Y$. A somewhat simpler
version of the invariant can be given by setting $U=0$, that is,
considering the free group generated over $\Field$ by $\Ta \cap \Tb$
and counting holomorphic disks only if $n_w (\phi )=0$, that is, the
holomorphic disks avoid the divisor $V_w$.

More generally, fix an integer $n\geq 1$, consider the chain complex
$(\CFm (Y), \partial ^-)$ and define
\[
(\CFm _{[n]} (Y), \partial ^-)=(\CFm (Y)/(U^n=0), \partial ^-).
\]
The boundary operator on the quotient is denoted by the same symbol as
on $\CFm (Y)$; notice that the operator on the quotient counts
holomorphic disks only when these pass through the basepoint at most
$n-1$ times.  We denote the homology groups of this quotient complex
by $\HFn (Y)$.  (In particular, it follows from the definition that
$\HFm _{[1]}(Y)=\HFa (Y)$.) It is easy to see that the truncated
groups $\HFn (Y)$ are diffeomorphism invariants of $Y$.  There is a
splitting of Heegaard Floer homology indexed by $\SpinC$ structures,
and in fact, this splitting is also a topological invariant of $Y$. We
will not treat this $\SpinC$ grading in the present paper.

The construction of Heegaard Floer homology can be extended to more general
Heegaard diagrams, where the handle decomposition results from a Morse
function with potentially more than one maximum and minimum. More specifically:

\begin{defn}
{\rm Let $\Sigma$ be an oriented surface of genus $g$, equipped with a
  $(g+k-1)$-tuple of disjoint, embedded $\alpha$--curves, the same
  number of disjoint, embedded $\beta$--curves, with the properties
  that
\begin{itemize}
\item the $\alpha$-curves span a 
$g$--dimensional subspace of $H_1 (\Sigma ; \Z)$, and
\item the $\beta$-curves span another 
$g$--dimensional subspace of $H_1 (\Sigma ; \Z)$.
\end{itemize}
This data specifies a three-manifold $Y$, and the tuple
$(\Sigma,\alphas,\betas)$ is called an
{\emph{extended Heegaard diagram for $Y$}}.
Suppose that we are given an extended Heegaard diagram as above so that
we can also find
$k$ basepoints $w_1,\ldots,w_k\in\Sigma-\alphas-\betas$,
satisfying the following conditions:
\begin{itemize}
\item each of the $k$ components of $\Sigma-\alphas$ contains exactly
  one of the basepoints in $\ws = \{w_i\}_{i=1}^k$ and 
\item each of the $k$ components of $\Sigma-\betas$ also contains
  exactly one of the basepoints $\ws = \{w_i\}_{i=1}^k$.
\end{itemize}
Then, we call the tuple $(\Sigma,\alphas,\betas,\ws)$ a 
\emph{multi-pointed Heegaard diagram for $Y$}.}
\end{defn}
Given a multi-pointed Heegaard diagram, we consider the chain
complex over $\Field[U]$ as before, endowed with the differential from
Equation~\eqref{eq:DefineDifferential}, where we use $n_w(\phi )$ to signify
the sum $\sum _{i=1}^k n_{w_i}(\phi)$, compare~\cite{Links}. We denote this
complex by $\CFmp(Y)$.

\begin{lem}
  \label{lem:CalculateCF2}
  The homology groups of $\CFmp(Y)$ determine $\HFm(Y)$ by the formula
  \[
  H_*(\CFmp(Y), \partial ^-)\cong \HFm(Y)\otimes H_*(T^{k-1}),
  \]
  where $T^{k-1}$ is the torus of dimension $k-1$, with $k$ denoting
  the number of basepoints.  Similarly, for any integer $n\geq 1$, we
  have
  \[H_*(\CFmp(Y)/(U^n=0), \partial ^-)
  \cong 
  \HFn(Y)\otimes H_*(T^{k-1}).
  \]
\end{lem}

\begin{proof}
  In~\cite[Theorem~4.4]{Links} it is shown that $\HFm(Y)$ can be
  calculated by a chain complex of the above type, except where there
  are $k$ different variables $U_i$ corresponding to the different
  basepoints. In the present context,  we implicitly set all these
  variables equal to one another. The first isomorphism now follows
  from simple homological algebra (compare
  \cite[Lemma~2.12]{MOST}). The second isomorphism follows similarly.
\end{proof}

\subsection*{Maslov index formulae}

In \cite{Lip}, Lipshitz developed a ``cylindrical reformulation'' of
Heegaard Floer homology, where his differentials count
certain pseudo-holomorphic curves in $[0,1]\times\R\times\Sigma$ satisfying
appropriate asymptotic conditions, in place of holomorphic disks in
the $g$-fold symmetric product of $\Sigma$. Using this reformulation,
he deduces concrete formulae for the Maslov index of a homology
class of Whitney disks in terms of combinatorial data on the Heegaard surface.
We recall these here.

Suppose for simplicity that in our Heegaard diagram every elementary
domain is a polygon.  (By isotoping the $\beta$--curves this property
can be easily achieved for any Heegaard diagram, cf. \cite{SW}.)  Then
for every component $D$ we can associate its \emph{Euler measure}
$e(D)$, which is equal to $1-\frac{m}{2}$ if $D$ is a $2m$--gon. This
quantity extends to formal linear combinations of the components: if
$\phi \in \pi _2 (\x, \y) $ is a given homology class of Whitney
disks, then the Euler measure $e(\phi )$ of $\phi$ is equal to
$e(D(\phi ))= e (\sum n_i D_i)=\sum n _i e(D_i)$.  The class $\phi \in
\pi _2 (\x , \y)$ also admits a \emph{point measure}, which is defined
as follows: Consider a corner point $x_i$ of $D(\phi)$, i.e. a
coordinate of $\x$ or $\y$. The average of the multiplicities of the
four domains of $\Sigma - \alpha - \beta$ in $D (\phi )=\sum
n_{d_i}(\phi )D_i$ meeting at $x_i$ is the {\em local multiplicity of
  $\phi$ at $x_i$}, and the sum of these quantities for all
coordinates of $\x$ and $\y$ gives the point measure $p(\phi )$ of
$\phi \in \pi _2 (\x , \y)$.  It turns out that the Maslov index $\mu
(\phi )$ can be conveniently expressed in terms of the Euler and point
measures as follows:
\begin{thm}(\cite[Corollary~4.3]{Lip}) \label{thm:lip1}
If $\phi \in \pi _2 (\x , \y )$ is a non-negative domain then $\mu (\phi )
= e(\phi ) + p (\phi )$. \qed
\end{thm} 
For $\phi \in \pi _2 (\x , \y )$ let $\Delta (\phi )$ denote the
intersection number of $\phi $ with the diagonal $\Delta \subset \Sym
^g (\Sigma _g )$.  A formula similar to the one in
Theorem~\ref{thm:lip1} relates the Euler measure $e(\phi )$ to the
Maslov index $\mu (\phi )$ with the help of the intersection number
$\Delta (\phi ) $:

\begin{thm}(\cite{Lip}) \label{thm:lip2}
For a non-negative domain $\phi \in \pi _2 (\x, \y )$ we have $\mu
(\phi ) = \Delta (\phi ) + 2e(\phi ).$ \qed
\end{thm}

As we have already mentioned, the above results follow from the
``cylindrical reformulation'' of Heegaard Floer theory. We state
another consequence of this theory which is of particular importance
to us.

\begin{thm}(\cite[Lemma~4.1]{Lip}) \label{thm:lip3}
  If $\phi\in\pi_2(\x,\y)$ is a non-negative homology class of Whitney
  disks, then there is an associated surface $T$ decorated with $2g$
  boundary points $\{p_i,q_i\}_{i=1}^g$ and a smooth
map $u\colon T \longrightarrow \Sigma$ satisfying the following
boundary conditions:
  \begin{itemize}
    \item The endpoints of the components of $\partial T-\{ p_i, q_i
      \}_{i=1} ^g$ are labelled with some $p_i$ and $q_i$ in an
      alternating manner.
    \item 
      The image of $\{p_1,\dots,p_g\}$ is the tuple $\x$, while the
      image of $\{q_1,\dots,q_g\}$ is the tuple $\y$.
    \item
      Each component $A$ of $\partial T-\{p_i,q_i\}_{i=1}^g$ which is
      oriented (with respect to the boundary orientation coming from
      $T$) as an arc from some $p_i$ to some $q_j$ is mapped into some
      $\alpha_i$.
    \item
      Each component $B$ of $\partial T-\{p_i,q_i\}_{i=1}^g$ which is
      oriented (with respect to the boundary orientation coming from
      $T$) as an arc from some $q_i$ to some $p_j$ is mapped into some
      $\beta_i$.
  \end{itemize}
  Moreover,
  \begin{itemize}
    \item The map $u$ represents the two-chain associated to $\phi$.
    \item 
      Let $d$ denote the number of coordinates of $\x$ where the local
      multiplicity of $\phi$ is non-zero; the remaining $g-d$
      coordinates of $\x$ coincide with $g-d$ coordinates of $\y$.
      Then, our source surface $T$ decomposes as the disjoint union
      $T=S\coprod P$, where $P$ consists of a collection of $g-d$
      disks which are mapped to $\Sigma$ via constant maps.  The Euler
      characteristic of $S$ is determined by the homology class of
      $\phi$ by the formula
\begin{equation}\label{e:euler1}
\chi (S)=d-\Delta (\phi ).
\end{equation}
\item The map $u$, when restricted to $S$,  is a branched cover onto its image.

\end{itemize}
\end{thm}

The above theorem is a consequence of the main result of~\cite{Lip},
where (for suitable almost-complex structure) the pseudo-holomorphic
disks in $\Sym ^g(\Sigma)$ are identified with pseudo-holomorphic maps
$$v\colon S \longrightarrow [0,1]\times\Sigma\times \R.$$ The map $u$
referred to in the statement of  Theorem~\ref{thm:lip3}
is obtained from the map $v$ by post-composing with the projection to
$\Sigma$. 

Note that the topology of $S$ is not (necessarily) uniquely determined
from the homology class of $\phi$; however, according to 
Equation~\eqref{e:euler1}, its Euler characteristic is.
Moreover, by Theorem~\ref{thm:lip2} this expression can be turned into
\begin{equation}\label{e:euler2}
\chi (S)=d-\mu (\phi )+2e(\phi ).
\end{equation}

Let $b$ denote the number of branch points of the branched cover map
$u$ when restricted to $S$. (Branch points on the boundary are counted
with multiplicity $\frac{1}{2}$ each.) This number then can be computed 
by the following formula of \cite[Proposition~4.2]{Lip}:
\begin{equation}\label{e:branch}
b=\mu (\phi )-e(\phi ) -\frac{1}{2}d.
\end{equation}
\begin{rem}\label{r:sharp}
{\rm Notice that when $d=2p(\phi)$ then the above formula (together
  with Theorem~\ref{thm:lip1}) implies $b=0$, hence in this case the
  map $u$ is a local diffeomorphism.}
\end{rem}

\section{Heegaard diagrams}
\label{sec:HeegaardDiagrams}
In the present section, we describe the class of Heegaard
diagrams used to establish Theorem~\ref{t:main}. After introducing
these diagrams, we collect some of their salient features.

A classical result of Hilden and Montesinos \cite{Hilden, Monte}
states that any closed, oriented 3--manifold $Y$ can be presented as a
simple 3--fold branched cover of $S^3$. Fix such a branched cover $\pi
\colon Y \longrightarrow S^3$, and let $R_Y\subset S^3$ denote its
branch set (or ramification set). The link $R_Y$ (as any link in
$S^3$) admits a \emph{grid presentation}, cf. \cite{BirmanMenasco,
  Brunn, Cromwell}.  This means that $S^3$ admits a genus--one
Heegaard diagram with $n$ parallel curves $\as=\{{\mathfrak
  a}_i\}_{i=1}^n$ specifying one handlebody, and $n$ parallel curves
$\bs=\{{\mathfrak b}_i\}_{i=1}^n$ orthogonal to the ${\mathfrak a}_i$,
and with two distinguished points (one of them usually denoted by an
$O$, the other one by an $X$) in every component of $T^2-\as$ and of
$T^2-\bs$. (In a subsequent construction we will enlarge both sets
$\as$ and $\bs$, but always so that curves in $\as$ are embedded and
pairwise disjoint, and so are curves in $\bs$.)  The link can be
recovered from this picture by connecting the $O$'s to the $X$'s in
one handlebody and the $X$'s to the $O$'s in the other. Representing
the torus by a square with its opposite sides identified, we can
picture this on a square grid with an $O$ and an $X$ in each row and
column, and $R_Y$ is given by connecting the letters in the same row
or column, with the convention that the vertical segment crosses over
the horizontal one. Note that our notation here has slightly departed
from the notation from Section~\ref{sec:Preliminaries}, where $\as$
would have been denoted $\alphas$ and $\bs$ would have been denoted
$\betas$. We do this to reserve $\alphas$ and $\betas$ for the
Heegaard diagram for $Y$, rather than $S^3$. We denote the set of all
the $O's$ and $X's$ by $\Os$ and $\Xs$, respectively.

The simple triple branched cover along $R_Y$ in its grid presentation
gives a multi-pointed Heegaard diagram for $Y$ as follows.  Let
$\Sigma\subset Y$ denote the inverse image of the Heegaard torus of
$S^3$ under the triple branched covering. Since $\Sigma$ is a simple
3--fold branched cover of the torus in $2n$ points, its genus is
$n+1$.  Since each $\as$-curve in the torus $T^2 \subset S^3$ bounds a
disk in the corresponding handlebody, it lifts to three disjoint
copies in $\Sigma$ (and similarly for the $\bs$--curves).  An annular
region $A_{i,i+1}$ between two neighboring $\as$--curves ${\mathfrak
  a}_i$ and ${\mathfrak a}_{i+1}$ in the grid diagram (containing two
branch points) lifts to a two--component surface: one of the
components is an annulus and the other is a 4--punctured sphere, as in
Figure~\ref{fig:branch1}.
\begin{figure}[ht]
\begin{center}
\input{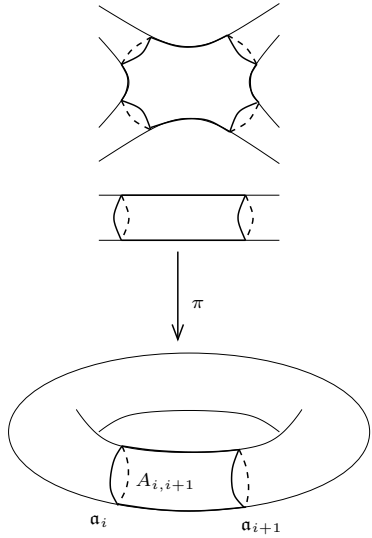}
\end{center}
\caption{Domains in the Heegaard decomposition of the triple branched
  cover.}
\label{fig:branch1}
\end{figure}
For every such region $A_{i,i+1}$ in the torus $T^2\subset S^3$ we
introduce a new curve ${\mathfrak a}_{i+\frac{1}{2}}\subset A_{i,i+1}$, which we
also include in the set of curves $\as$, and which is parallel
(i.e. isotopic to, and disjoint from) both ${\mathfrak a}_i$ and ${\mathfrak a}_{i+1}$, but
separates the two branch points in $A_{i,i+1}$, cf. Figure~\ref{f:new}.  Notice that
there are two essentially different choices for such a curve,
depending on the side of the new curve $X$ (or $O$) is positioned.  We
can choose ${\mathfrak a}_{i+\frac{1}{2}}$ to be either of these two possible
curves.
\begin{figure}[ht]
\begin{center}
\setlength{\unitlength}{1mm}
\input{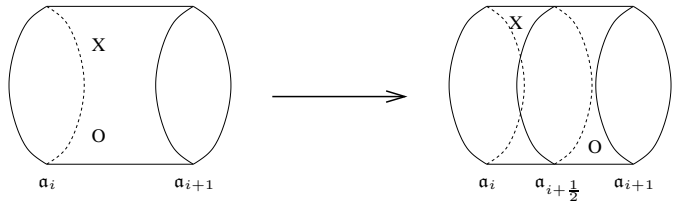}
\end{center}
\caption{Introduction of a new $\alpha$--curve.}
\label{f:new}
\end{figure}
The inverse image of this new curve in $\Sigma$ has
two components: one of these components is in the annular component of
$\pi ^{-1}(A_{i,i+1})$, while the other one 
(which double--covers ${\mathfrak a}_{i+\frac{1}{2}}\subset T^2$) is
in the 4--punctured sphere component. This latter component
obviously
separates the 4--punctured sphere into two pairs of pants,
cf. Figure~\ref{fig:branch2}.  
\begin{figure}[ht]
\begin{center}
\setlength{\unitlength}{1mm}
\input{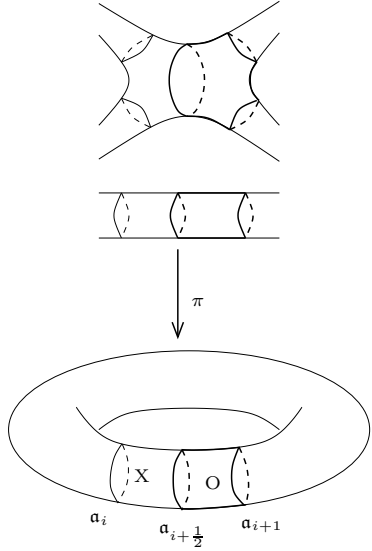}
\end{center}
\caption{Domains in the Heegaard decomposition of the triple branched
  cover after introducing the new curves.}
\label{fig:branch2}
\end{figure}
We also add the similar half--indexed curves in the
$\bs$--direction. These newly chosen curves (as well as the
corresponding circles in the Heegaard diagram for $Y$) will be
referred to as \emph{new curves in the grid diagram}, while the $\as$-- and
$\bs$--curves in the original grid diagram for $R_Y$ 
will be the \emph{old curves}.  We will call the grid
diagram together with the new curves an \emph{extended grid
  diagram}.

Consider now the extended grid diagram for $R_Y$. Both $\as$ and $\bs$
partition the torus into $2n$ annuli, each annulus containing a single
branch point, resulting in a $2n$--pointed Heegaard diagram for
$S^3$. Clearly, every elementary domain in this diagram is a
rectangle. Consequently, in the branched cover
$(\Sigma,\pi^{-1}(\as),\pi^{-1}(\bs))$ of the extended grid diagram
(thought of as an extended Heegaard diagram for $Y$) every elementary
domain is either an octagon (if it double covers an elementary
rectangle with one $X$ or an $O$ in it), or otherwise it is a
rectangle.  To get a multi-pointed Heegaard diagram for $Y$ (as
required in the construction for Heegaard Floer homology), we have
place basepoints in some of the elementary domains. By simply putting
the basepoints into the octagons in $\Sigma$ we do not get a
multi-pointed Heegaard diagram, since not every component of the
complement of the preimages of the $\as$-- or the $\bs$--curves
contains a basepoint in this way: the annular components of $\pi
^{-1}(A_{i,i+1})$ do not contain octagons, hence do not contain
basepoints either. To fix this problem we could either put in more
basepoints, or we could delete some curves in $\pi^{-1}(\as)$ and
$\pi^{-1}(\bs)$ to specify $\alphas$ and $\betas$.  We choose to do
the latter, by applying the following conventions: First, choose
coherent orientations for the ${\mathfrak a}_i\in\as$ and the
${\mathfrak b}_j\in\bs$, as specified by a basis $A,B\in H_1(T^2)$;
i.e., for all $i=1,\dots,n$, curves ${\mathfrak a}_i$ and ${\mathfrak
  b}_i$ (with suitable orientations) represent the homology class $A$
and $B$ respectively.  Such a choice also induces a cyclic ordering on
the curves in $\as$: the curve ${\mathfrak a}_{i+1}$ follows
${\mathfrak a}_i$ if the orientation on ${\mathfrak a}_{i+1}$, thought
of as a boundary component of $A_{i,i+1}\subset T^2$ with its induced
orientation, agrees with the orientation specified by $A$.  The curve
${\mathfrak {a}}_{i+1}$ will be also called the {\em right} endcircle
of the annulus $A_{i,i+1}$.  We have an analogous cyclic ordering on
the $\bs$--curves. With this orientation convention in place, we refine our
definition of a Heegaard diagram for $Y$ as follows:

\begin{enumerate}
\item For a new curve ${\mathfrak a}_{i+\frac{1}{2}}$ we 
  consider
  only the component $\alpha_{i+\frac{1}{2}}$ of the inverse image of
  ${\mathfrak a}_{i+\frac{1}{2}}$ in $\Sigma$ which double covers the curve
  downstairs (i.e., the one separating the 4--punctured sphere
  component into two pairs of pants). Similar choice applies for every
  new $\beta$--curve.
\item For an old $\as$--curves we keep as $\alphas$-curves 
  only the two components of
  the inverse image of the \emph{right} endcircle of the annulus $A_{i,i+1}
  \subset T^2 \subset S^3$ which are in the 4--punctured sphere (and
  do not consider the one in the annular component in $\pi
  ^{-1}(A_{i,i+1})$). We apply the same principle to the
  $\beta$--circles.  
\end{enumerate}

\begin{thm}
  \label{thm:Adapted}
  The construction of $\Sigma$, $\alphas$, $\betas$ (as specified by
  the extended grid diagram $\as$, $\bs$, and the basis $A, B\in
  H_1(T^2)$) together with a choice of $k$ basepoints, one in each
  octogonal elementary domain, provides a multipunctured Heegaard
  diagram for $Y$.
\end{thm} 
\begin{proof}
The pull-back of the grid diagram from $S^3$ along $\pi $ obviously
defines a Heegaard diagram of $Y$ with Heegaard surface of genus $n+1$
and $3n$ $\alpha$-- and $\beta$--curves. Since the inverse images of
the new curves all bound disks in the appropriate handlebody of the
Heegaard decomposition of $Y$, by adding these curves we still have an
extended Heegaard diagram of $Y$ (now with $5n$ $\alpha$-- and
$\beta$--curves). Since we deleted only curves which in homology
linearly depended on the remaining curves, after the deletion we still
have a Heegaard diagram for $Y$.  The number of $\alpha$--curves (and
$\beta$--curves) in the resulting diagram is equal to $3n$, hence
placing $2n$ basepoints we get a multi-pointed Heegaard diagram. The
algorithm for deleting the curves is designed in such a way that the
collection of the $\alpha$--curves (and similarly the $\beta$--curves)
gives a pair of pants decomposition of the Heegaard surface
$\Sigma$. Since by construction each pair of pants contains a unique
octagon, and we placed the basepoints exactly in the octagons, it
follows that the result is a multi-pointed Heegaard diagram for $Y$.
\end{proof}

\begin{defn} 
{\rm Given a simple 3--fold branched cover $\pi \colon Y
  \longrightarrow S^3$, fix an extended grid diagram
  $(T^2,\as,\bs,\Os,\Xs)$ for the branch set $R_Y$ of $\pi$ in $S^3$,
  and a basis $A, B$ for $H_1(T^2; \Z)$ representing $\as$ and $\bs$
  respectively.  According to Theorem~\ref{thm:Adapted}, these choices
  determine a multi-pointed Heegaard diagram
  $(\Sigma,\alphas,\betas,\ws)$ for $Y$, which we call \emph{adapted
    to the choices} or, when we wish to suppress the choices, simply
  an \emph{adapted Heegaard diagram for $Y$}.  In an adapted Heegaard
  diagram an $\alpha$- or a $\beta$-curve is called a \emph{new curve
    in $\Sigma$} if it projects onto a new curve in the extended grid
  diagram; otheriwse it is called an \emph{old curve in $\Sigma$}.
  Sometimes, we drop the qualifier, and refer simply to \emph{new
    curves} and \emph{old curves}, when $\Sigma$ is to be understood
  from the context.}
\end{defn}

Combining  the above construction with the theorem of Hilden and Montesinos, 
we obtain the following:
\begin{cor}\label{c:van}
Any closed, oriented 3--manifold admits an \adapted Heegaard diagram.
\qed
\end{cor}

The next proposition collects the most basic properties of \adapted Heegaard
diagrams. Further, more delicate properties of these diagrams (relevant to our
subsequent discussions) will be discussed afterwards.

\begin{prop}\label{p:Properties}
An adapted Heegaard diagram for $Y$ satisfies the following
properties:
\begin{enumerate}
\item Every elementary domain in the Heegaard diagram is either a
  rectangle, which does not contain a basepoint, or it is an octagon,
  in which case it does contain a unique basepoint.
\item Both the $\alpha$-- and $\beta$--curves give pair of pants
  decompositions of $\Sigma$; and each pair of pants contains a unique
  elementary domain which is an octagon (and all others are
  rectangles).
\item Any two $\alpha$-- and $\beta$--curves meet either exactly once
  (this is guaranteed if at least one is an old curve), or exactly
  twice (which happens only when both are new curves), and in this
  latter case the two intersections come with equal signs.
\item \label{property:ProjectGenerator} If $\x$ is a generator of
  $\CFm (\Sigma , \alpha , \beta )$, then any curve ${\mathfrak
  a}\in\as\subset T^2$ meets the projection $\pi (\x)$ of $\x$ in one
  or two points (counted with multiplicity). In the former case
  ${\mathfrak a}$ is a new curve while in latter case it is an old
  curve in the extended Heegaard diagram. Similar statement applies to
  the ${\mathfrak {b}}$--curves.
\item \label{property:Inverses} Over every intersection point of some
  ${\mathfrak a}_i\in\as$ and ${\mathfrak b}_j\in \betas$ in the
  extended grid diagram there is either one or two intersection
  points of some $\alpha_k$ and $\beta_\ell$ in the \adapted Heegaard
  diagram. 
\item The Heegaard diagram is admissible.
\end{enumerate}
\end{prop}
\begin{proof}
 The first five properties obviously follow from the definition of the
 Heegaard diagram. Notice that by Property (1) the Heegaard diagram we
 defined is nice in the sense of \cite{SW}.  Property (6) then follows
 from~\cite[Corollary~3.2]{LMW}.
\end{proof}

Next we discuss several further properties of this Heegaard diagram,
which will be used in our arguments. In the proofs we will heavily use
the extra structure an \adapted Heegaard diagram admits, i.e.  the
extra restrictions imposed by the existence of the projection map from
$(\Sigma , \alphas , \betas )$ to the extended grid diagram.

\begin{lem}
  \label{lem:ThereIsABasepoint}
  Let $a\subset \alpha_i$ be an arc in a new curve $\subset \Sigma$
  with the property that $\pi$ restricted to $a$ is not injective.
  Then, there is an elementary domain on either side of $a$ which
  contains a baspoint.
\end{lem}

\begin{proof}
  By the assumption, the projection $\pi (a)$ is a full circle in the
  torus, i.e., the homology class $[\pi (a)]$ is equal to $\pm A $ (or
  $\pm B$). Therefore $\pi (a)$ has a rectangle on either of its sides
  which contains a basepoint. These rectangles lift to octagons
  containing basepoints as stated.
\end{proof}

\begin{lem}
  \label{lem:TwoNewCurves}
  Let $\phi\in\pi_2(\x,\y)$ be a non-negative homology class, and let
  $D=D(\phi )$ be its corresponding domain.  Suppose that $a_1,
  a_2\subset \partial D$ are two arcs in $\partial D$ which are
  contained in new curves $\alpha_1$ and $\alpha_2\subset \Sigma$
  respectively, and which have the property that the projection map
  $\pi$ restricted to both $a_1$ and $a_2$ is not injective.  Then,
  $n_w(\phi)\geq 2$.
\end{lem}

\begin{proof}
  By Lemma~\ref{lem:ThereIsABasepoint} there are basepoints $p_1$ and
  $p_2$ in $D$ adjacent to both $a_1$ and $a_2$, in the sense that
  there are elementary domains $D_1$ and $D_2$ appearing with non-zero
  multiplicity in $D$, and $D_i$ contains $p_i$ and also an arc in
  $a_i$ on its boundary.  Since in the extended grid there is no
  elementary domain which has two arcs on its boundary from different
  new $\alpha$--curves, the same holds in the \adapted Heegaard
  diagram of $Y$. This shows that $D_1\neq D_2$, implying that $p_1$
  and $p_2$ are distinct, concluding the proof.
\end{proof}

There is a principle which is useful for excluding the existence of
certain homology classes in an \adapted Heegaard diagram, with
specified local multiplicities at their corners. Since similar
arguments will be repeatedly used in our latter discussions, we
illustrate this principle in a particular case.

\begin{lem}\label{lem:Principle}
  There is no homology class of Whitney disks $\phi\in\pi_2(\x,\y)$
  with the following properties:
  \begin{itemize}
        \item The homology class $\phi$ is non-negative,
        \item $n_w(\phi)\leq 1$,
    \item 
      there are exactly two $\x$-- and two $\y$--coordinates ($x_1,
      x_2$ and $y_1, y_2$) where the local multiplicity is non-zero,
    \item the local multiplicities at these four points are
      $\{\frac{1}{4}, \frac{1}{4}, \frac{3}{4}, \frac{3}{4}\}$ in some
      order,
    \item the points $\{x_1,y_1,x_2,y_2\}$ project to four distinct
      points on the torus, and
    \item all four points $x_1,y_1,x_2,y_2$ lie on a single component
      in the boundary of $D$.
  \end{itemize}
\end{lem}
\begin{proof}
  The local multiplicites in the four quadrants around a point of
  local multiplicity $\frac{3}{4}$ can distribute in two fundamentally
  different ways. They can be either $\{0,1,1,1\}$, or they can be
  $\{0,0,1,2\}$. (The third combinatorial possibility, $\{ 0,0,0,3\}$,
  cannot occur since $\phi $ is a homology class of Whitney disks.)
  We exclude first the possibility that all local multiplicities
  around each corner point are $\leq 1$, i.e. the case where the four
  multiplicities distribute as $\{0,1,1,1\}$ around both points with
  multiplicity $\frac{3}{4}$.

  We construct a closed, embedded, oriented path $\gamma$ in the torus
  which goes through the projection of the four given points,
  consisting of arcs among the projections of the $\alpha_i$ and
  $\beta_j$, and turning $90^\circ$ left or right depending on whether
  the corner point required to have local multiplicity $\frac{1}{4}$
  or $\frac{3}{4}$, respectively. There are two combinatorially
  different cases for $\gamma$, according to the cyclic order of the
  different local multiplicities; these possibilities are illustrated
  in Figure~\ref{fig:CantFindNullHomology}. (The grid torus is given
  by the usual identification of opposite sides of the rectangles
  drawn.)
\begin{figure}[ht]
\begin{center}
\setlength{\unitlength}{1mm}
{\input{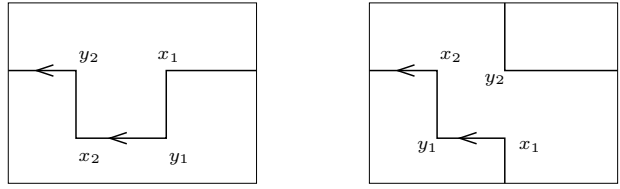}}
\end{center}
\caption{Choices of initial curve $\gamma$.}
\label{fig:CantFindNullHomology}
\end{figure}

Note that $\gamma$ does not {\em a priori} have to agree with the
projection of the boundary of our domain $D$. However, it is easy to
see that the projection of the boundary of $D$ is necessarily gotten
from $\gamma$ by the following procedure: given any of the four
turning points of $\gamma$, if we approach that corner of $\gamma$
along some arc, then we can elongate that arc so that it goes around
the torus (with the same orientation) more times.

As usual, let $A$ and $B$ denote two generators of $H_1(T^2; \Z )$,
where $A$ is the homology class of the projection of an old curve
$\alpha_i$ (with suitable orientation) while $B$ is the homology class
of the projection of an old curve $\beta_j$.  The possible choices of
$\gamma$ illustrated in Figure~\ref{fig:CantFindNullHomology} show
that $\gamma$ represents either the homology class $A$, $B$ (note that
we are free to rotate the picture $90^\circ$), or $A+B$.

Consider the case where the homology class of $\gamma$ is $A$. In this case,
our allowed modification can add on some number of copies of $\pm B$, or
alternatively some (non-negative) number of copies of $A$ (but no copies of
$-A$).  Thus, this modification can never end up giving a null-homologous
curve; on the other hand, the boundary of our domain must be null-homologous.

Similarly, in the case where the homology class of our curve $\gamma$ is $A+B$,
we can add on more copies of $A$ and $B$ (but not $-A$ or $-B$). Once
again, we can never do this to achieve a curve which is null-homologous.

We have just excluded the possibility that the local multiplicities
around both of the $\frac{3}{4}$ points are all $\leq 1$
(and indeed, for that part of the argument, we have not
yet used the hypothesis that $n_w(\phi)\leq 1$). Suppose now that at least one
of these corner points has local multiplicities $\{0,0,1,2\}$.  It
follows readily that the other $\frac{3}{4}$ point has the same
distribution of local multiplicities: the arc separating the region
with multiplicity $2$ from the adjacent region with multiplicity $0$
appears with multiplicity $2$ in the boundary, and this phenomenon
does not occur for any other corner point whose local multiplicities
are all $\leq 1$ (but non-negative). 
Thus, it follows in fact that the two $\frac{3}{4}$ corners
share an edge. That edge might be of type $\alpha$ or $\beta$:
the two cases are the same (with a little change of notation), so we
assume that latter.

In fact, we label coordinates as in Figure~\ref{fig:Case0012}.
The torus is divided into four regions, labelled $W$, $X$, $Y$, and
$Z$. We label the adjoining curves in the extended Heegaard diagram
$a_1$, $a_2$, $b_1$, $b_2$ (without regard to the earlier
half-integral labeling).  The points $x_1$ and $y_1$ have local
multiplicity $\frac{3}{4}$, which distribute as $\{0,0,1,2\}$. In
fact, the region around $x_1$ and $y_1$ with mulipliticty $2$ project
to $W$, while the region with multiplicity $1$ around $x_1$ and $y_1$
project to $X$.

Let $a\subset \alpha_1$ be the arc from $x_1$ to $y_2$. We claim that
this surjects onto ${\mathfrak a}_1$ in $T^2$. To see this, let $\xi$
be the arc which is the top edge of $W$: it connects $\pi(x_1)$ to
$\pi(y_2)$.  Let ${\widetilde \xi}$ be the lift of $\xi$ to $\Sigma$,
so that its initial point is $x_1$, and we let its final point be
denoted $y$.  The final point $y$ is adjacent to a region with
multiplicity $2$ (the region projecting to $W$). It follows that $y$
cannot be a corner point (the only corner points with multiplicity $2$
in an adjacent region are $x_1$ and $y_1$), and hence $a$ must be an
arc in a new curve. Thus, the arc in the boundary of $D(\phi)$
connecting $x_1$ to $y_2$ must project to an arc in $T^2$ which covers
the projection of $x_1$ twice. The same argument works for the arc connecting $y_1$ and
$x_2$. In conclusion we found two arcs in $\partial D(\phi )$ which
are contained by new $\alpha$--curves and on which $\pi $ is not
injective. The application of Lemma~\ref{lem:TwoNewCurves} now
provides the desired contradiction with the fact that $n _w(\phi )\leq
1$.
\end{proof}
\begin{figure}[ht]
\begin{center}
\setlength{\unitlength}{1mm}
{\input{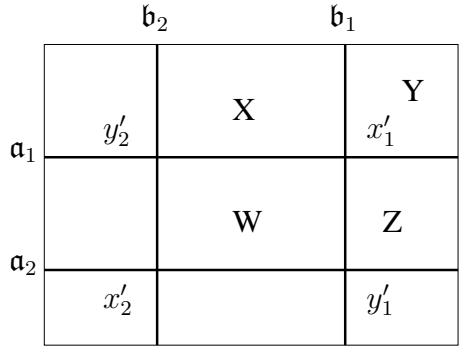}}
\end{center}
\caption{{\bf Notation for the end of Lemma~\ref{lem:Principle}}.
Here, we let $x_i'$ be the projection of $x_i$ and $y_i'$ be the
projection of $y_i$.}
\label{fig:Case0012}
\end{figure}

\section{Non-negative domains with Maslov index one}
\label{sec:fifth}

We can now formulate a more precise version of Theorem~\ref{t:main},
showing that $\CFtwo(Y)$ can be determined combinatorially using an
adapted Heegaard diagram for $Y$. The proof of the following theorem
will rely on results of this and the subsequent sections and will be
given at the end of Section~\ref{sec:CalcContributions}.

\begin{thm}\label{t:contribution}
  Suppose that $Y$ is given by an \adapted Heegaard diagram.  Then for
  any non-negative homology class $\phi$ with Maslov index $\mu (\phi
  )=1$ and $n_w(\phi) =\sum n_{w_i}(\phi) \leq 1$, and for any choice
  of almost-complex structure $J$, the number $c(\phi,J)
  =\#\UnparModFlow(\phi)$ is equal to 1 modulo 2.
\end{thm}

The proof will start with the analysis of the possible non-negative
domains with $n_w(\phi)\leq 1$ and with Maslov index zero. We will
show that in an \adapted Heegaard diagram, if $\phi$ satisfies these
hypotheses, then it represents a constant map.  From this (as it is shown 
by Proposition~\ref{prop:NoIndexZeros}) it will follow that the
contribution $c(\phi,J)$ of any homology class of Maslov index one to
the differential is independent of the chosen almost--complex
structure $J$. Afterwards, in Proposition~\ref{prop:OnlyIndexOnes}, we
place strong topological restrictions on the possible homology classes
of non-negative Whitney disks with Maslov index one and $n_w(\phi)\leq
1$. Indeed, with a little more work in
Section~\ref{sec:CalcContributions}, we are able to classify such
domains, as stated in Theorem~\ref{thm:ClassifyDomains}.  We can then
verify that the contributions of these domains are equal to 1 (mod 2).
We begin with the promised analysis of Maslov index zero domains.

\begin{prop}
\label{prop:NoIndexZeros}
Suppose that $Y$ is given by an \adapted Heegaard diagram. Then for any
non-negative domain $\phi\in \pi_2(\x,\y )$ with Maslov index $\mu
(\phi)=1$ and $n_w(\phi)=\sum n_{w_i}(\phi) \leq 1$ the number
$c(\phi,J) =\#\UnparModFlow(\phi)$ is independent of the chosen
almost-complex structure $J$.
\end{prop}

\begin{proof}
  The number $c(\phi,J)$ {\em a priori} might depend on the chosen
  almost-complex structure $J$. However, in that case, there must be a
  different, non-constant, non-negative homology class $\psi$ of
  Whitney disks with $\Mas(\psi)=0$ with $n_w(\psi)\leq n_w(\phi)$.
  This can be seen from Gromov's compactness theorem: suppose that
  $J_t$ is a generic one-parameter path of almost-complex structures
  connecting $J_1$ to $J_2$ (in fact, $J_1$ and $J_2$ are actually
  paths of almost-complex structures themselves, and hence $J_t$ is
  really a path of paths), and consider the parameterized moduli space
  consisting of pairs $(t,u)$ where $t\in [1,2]$ and $u\in
  \UnparModFlow_{J_t}(\phi)$. This moduli space has boundary at $t=1$,
  $t=2$ (counted by the difference $c(\phi,J_1)-c(\phi,J_2)$), and
  further ends at splittings parameterized by triples $t$, $u_1$,
  $u_2$, where
  \begin{itemize}
  \item $t\in (1,2)$,
  \item $u_1$ and $u_2$ are (non-constant)
    $J_t$-pseudo-holomorphic disks, and 
  \item $u_1$ and $u_2$ represent homology classes $\phi_1$ and
    $\phi_2$ of Whitney disks, so that the juxtaposition of $\phi_1$
    and $\phi_2$, $\phi_1*\phi_2$ represents $\phi$.
  \end{itemize}
  Moreover, by choosing our path generically, we can arrange that
  $\{\Mas(\phi_1),\Mas(\phi_2)\}=\{0,1\}$. Thus, if $c(\phi,J_1)\neq
  c(\phi,J_2)$, there must be an end corresponding to a decomposition
  of $\phi$. Our desired $\psi$ is then either $\phi_1$ or $\phi_2$
  (whichever has Maslov index equal to zero).
  
  We now analyze non-negative domains with Maslov index $0$ in our
  Heegaard diagram. We will show that such domain corresponds to a
  constant map, hence, according to the above said, we will verify
  that for each Maslov index one homology class of Whitney disks the
  quantity $c(\phi,J)$ is indepdent of $J$.

  To this end, recall that $\Mas(\psi ) = e(\psi )+p(\psi )$ and by
  our assumption $n_w(\psi ) \leq 1$. This implies that the domain of
  $\psi$ contains at most one octagon with multiplicity at most one,
  hence $e (\psi )\geq -1$.  Notice that in an \adapted Heegaard
  diagram every domain has integral Euler measure, hence $e(\psi)$ is
  also an integer.  If $\psi$ is non-constant, we have that $p (\psi
  )>0$.  Since $\Mas(\psi )=0 $ and $n_w(\psi )\leq 1$, it follows that
  $e(\psi )=-1$ and $p(\psi )=1$. Therefore a non-negative domain with
  $n_w(\psi )\leq 1$ and $\Mas (\psi )=0$ has $\Delta (\psi )=2$ and
  $d\leq 2$.
  
  In the case where $d=1$, the possible local multiplicities of $\x$
  and $\y$ are $(\frac{1}{4}, \frac{3}{4})$ and $(\frac{1}{2},
  \frac{1}{2})$. In both cases the projection of the image of the
  boundary represents a nontrivial homology class in $T^2$, leading to
  a contradiction as follows. In the first case, the distinct
  coordinates of $\x$ and $\y$ are the two intersection points of a
  pair of new $\alpha$-- and $\beta$--curves, hence the projection of
  the boundary represents $\pm A \pm B$ in $H_1(Y; \Z)$.  If the local
  multiplicities are given by $(\frac{1}{2},\frac{1}{2})$ then
  $\partial S$ maps to a complete $\alpha$-- or $\beta$--circle, hence
  the projection of the image of $\partial S$ represents $\pm A$ or
  $\pm 2A$ (or $\pm B, \pm 2B$) in the homology of the torus. In both
  cases, $\pi(u(\partial S))$ is not null-homologous, providing a
  contradiction to the possibility that $d=1$.
  
  In the case where $d=2$, the local multiplicities must be equal to
  $(\frac{1}{4}, \frac{1}{4}, \frac{1}{4}, \frac{1}{4})$ and the Euler
  characteristic of the surface $S$ (given as in
  Theorem~\ref{thm:lip3}) is $0$. Hence the domain must have two
  boundary components, each made out of an $\alpha$-- and a
  $\beta$--arc, intersecting each other twice.  Since the local
  multiplicities are all equal to $\frac{1}{4}$, the two points must
  be distinct, hence the boundary arcs are all in new curves and the
  projection map on any of the boundary arcs is not injective. This
  means that the assumptions of Lemma~\ref{lem:TwoNewCurves} are
  satisfied, implying $n_w(\psi )\geq 2$, contradicting our assumption
  $n _w (\psi )\leq 1$. In summary, we have showed that there is no
  non-constant, non-negative domain $\psi $ with $n_w(\psi )\leq 1$
  and $\mu (\psi )=0$ in our Heegaard diagram, which observation
  concludes the proof of the proposition.
\end{proof}

The next result provides a strong restriction on non-negative
domains with Maslov index one in \adapted Heegaard diagrams.

\begin{prop}
  \label{prop:OnlyIndexOnes}
  Let $\phi\in\pi_2(\x,\y)$ be any non-negative domain with Maslov
  index one and $n_w(\phi)\leq 1$. Then the associated domain $D$ of
  $\phi$ is the immersed image of either
\begin{enumerate}[label={(\Alph*)}]
    \item
      \label{case:Rectangle}
      a rectangle,
    \item
      \label{case:Annulus}
      a rectangle with a disk removed, or
    \item
      \label{case:Octagon}
      an octagon.
  \end{enumerate}
  Moreover, in Cases~\ref{case:Rectangle} and~\ref{case:Octagon}, the
  local multiplicities at all the coordinates of $\x$ and $\y$ are $0$
  or $\frac{1}{4}$, while in Case~\ref{case:Annulus} there is a single
  coordinate $p\in \x\cap\y$ where the local multiplicity is
  $\frac{1}{2}$, and at all other coordinates the local multiplicities
  are $0$ or $\frac{1}{4}$.
\end{prop}

\begin{proof}
Since all elementary domains which do not contain basepoints are
rectangles, the argument of \cite{SW} applies, showing that the domain
associated to any Maslov index one Whitney disk $\phi$ with
$n_w(\phi)=0$ is an embedded rectangle, and in fact all the local
multiplicities at $\x\cup \y$ are $\leq \frac{1}{4}$. (In fact, for
adapted diagrams this can be seen more directly; see the proof of
Proposition~\ref{prop:Rectangle} below.) This is
Case~\ref{case:Rectangle} above.

If we consider $n_w(\phi)=1$, then the Euler measure $e(\phi)$ is $-1$
(since the domain contains a unique octagon and some number of
rectangles), so since $\Mas(\phi)=1$, it follows that point measure
$p(\phi)$ must be equal to $2$.  From the equation $\Mas(\phi) =
\Delta (\phi)+2e (\phi)$ we deduce $\Delta (\phi)=3$.  Furthermore,
the Euler characteristic of the surface $S$ associated to $\phi$ (as
in Theorem~\ref{thm:lip3}) is given by Equation~\eqref{e:euler2} as
\[
\chi(S)=d-3.
\]
Since $p(\phi)=2$ implies that there are at most 8 coordinates in $\x
\cup \y$ with non-zero local multiplicities, we get that $d\leq 4$.
Consequently there are several cases to consider according to the
possible values of $d$:

\noindent {\bf Case I:} $d=1$

\noindent This case can be easily excluded, since by $d=1$ there are
only 2 coordinates with non-zero local multiplicity, hence the surface
$S$ has connected boundary. On the other hand,
Equation~\eqref{e:euler2} forces $\chi(S)$ to be $-2$ (implying the
existence of an even number of boundary components), giving a
contradiction.

\noindent {\bf Case II:} $d=2$

\noindent We claim that there are no domains with $\Mas(\phi)=1$,
$n_w(\phi)=1$, and $d=2$ in an \adapted Heegaard diagram. The argument
proceeds by a case-by-case check.  Notice first that when $d=2$, the
Euler characteristic of the surface $S$ is equal to $-1$, hence it has
an odd number of boundary components, which in this case can be only
one.  Since $d=2$, there are four corner points, and each comes with a
local multiplicity of at least $\frac{1}{4}$. Furthermore, the sum of
the local multiplicities is equal to $2$. Thus, the local
multiplicities can distribute over these four points according to the
following five possible ways:
\begin{list}
        {(II.\arabic{bean})}{\usecounter{bean}\setlength{\rightmargin}{\leftmargin\
}}
\item 
\label{k1}
$(1,\frac{1}{2},\frac{1}{4},\frac{1}{4})$
\item 
\label{k2}
$(\frac{3}{4},\frac{1}{2},\frac{1}{2},\frac{1}{4})$
\item 
\label{k3}
$(\frac{1}{2},\frac{1}{2},\frac{1}{2},\frac{1}{2})$
\item
\label{k4}
 $(\frac{3}{4},\frac{3}{4},\frac{1}{4},\frac{1}{4})$
\item
\label{k5} $(\frac{5}{4},\frac{1}{4},\frac{1}{4},\frac{1}{4})$.
\end{list}
Case~(II.\ref{k1}) does not exist, since if there is an
$\x$--coordinate with local multiplicity 1, then there should be a
similar $\y$--coordinate. Cases (II.\ref{k2}) and (II.\ref{k3}) are
impossible since a pair of $\x$-- and $\y$--coordinates with local
multiplicity $\frac{1}{2}$ require an entire boundary component, but
in our case the boundary is connected.

To exclude Case (II.\ref{k4}) first we show that the corner points
must project onto four distinct points on the torus.  Indeed, by
Property~\eqref{property:Inverses} of Proposition~\ref{p:Properties}
the corner points could {\em a priori} project to either two, three,
or four points. If the projection consists of two points, then it is
easy to see that both points lie on the same projected $\alpha_i$ or
$\beta_j$. Assume without loss of generality that they lie on the same
projected $\alpha_i$.  Then they lie on different projected $\beta_j$,
for which the projection $\pi $ is obviously not injective, and hence
Lemma~\ref{lem:TwoNewCurves} applies and provides a contradiction to
$n_w(\phi )\leq 1$. The case where the projection consists of three
points is excluded by Property~\eqref{property:ProjectGenerator} given
in Proposition~\ref{p:Properties}. Specifically, in this case it is
easy to see that there must be some $\alpha_i$ or $\beta_j$ whose
projection contains a corner point with multiplicity one, although
Property~\eqref{property:ProjectGenerator} cited above obviously
implies that the total multiplicity of corner points on the projection
of each $\alpha_i$ and $\beta_j$ must be even. We are left with the
possibility that the four corners project to four distinct points on
the torus. This case, however, is exactly the one handled by
Lemma~\ref{lem:Principle}, concluding the discussion of case
(II.\ref{k4}).

Case (II.\ref{k5}) can be excluded by analyzing the projection of the
boundary of the domain again.  First, observe that if we consider the
corner point of multiplicity $\frac{5}{4}$, then that must have some
adjoining region with local multiplicity at least $2$. Following the
region with multiplicity $2$ $\alpha$-curve, we arrive at another
corner which has no adjoining region with local multiplicity $2$; it
follows easily that the $\alpha$--curve had to be a new curve.
Similarly for the adjoining $\beta$--curve. Bearing this fact in mind,
we project the four corners down to the torus.  The corner points (as
before) project to either two, three, or four distinct points on the
torus. The case of two and three points can be excluded exactly as
before. Consider now the possibility that the corner points have four
distinct projections, which we denote by $\{\pi (x_1),\pi (y_1), \pi
(x_2) ,\pi (y_2) \}$.  Furthermore, let $a_1, a_2, b_1, b_2$ denote
the four sides of the image of the rectangle $S$ with a map $u$
representing $\phi$. According to Lemma~\ref{lem:TwoNewCurves} the
projection $\pi$ must be injective either on $a_1$ or on $a_2$ (and
similarly either on $b_1$ or on $b_2$). If $\pi$ on $a_1$ is not
injective (but it is on $a_2$), then the projection of the image of
$\partial S$ will have non-trivial $A$-component in the homology of
the torus, which would contradict the fact that this projection is
null-homologous. A similar observation applies to the $b$--sides,
implying at the end that the projection of the boundary of $u(S)$ is
an embedded (null-homologous) curve in $T^2$.  Thus, taking its
complement, we see that $\Sigma$ is divided into two regions where the
local multiplicities are locally constant. But at the four corners we
see at least three distinct local multiplicities, providing a
contradiction.

\noindent {\bf Case III:} $d=3$

\noindent In this case, the surface $S$ has vanishing Euler
characteristic, hence an even number of boundary components, which number 
(by $d=3$) is equal to two.  Thus, the surface is an annulus. On one
boundary component there are two corner points, and all the remaining
four are on the other component.  Regarding possible local
multiplicities, there are two cases to distinguish:

\noindent {\bf {(III.1)}}:
$(\frac{3}{4},\frac{1}{4},\frac{1}{4},\frac{1}{4},\frac{1}{4},\frac{1}{4})$,
and

\noindent {\bf {(III.2)}}:
$(\frac{1}{2},\frac{1}{2},\frac{1}{4},\frac{1}{4},\frac{1}{4},\frac{1}{4})$.

We discuss possibilities in {\bf {(III.1)}} first.  The boundary
component with two corners consists of two curves intersecting each
other twice, hence the arcs are in new curves and the projection map
$\pi$ is not injective on these arcs. In addition, their projection in
the torus represents the homology class $\pm A \pm B$. Now $\pi $ is
either non-injective on one of the arcs in the rectangular boundary
component (in which case Lemma~\ref{lem:TwoNewCurves} delivers a
contradiction to $n_w(\phi )\leq 1$), or $\pi $ is injective on the
rectangular boundary, consequently the projection of this boundary
component represents zero in the homology of the torus.  Since the
difference of the homology classes represented by the projections of
the two boundary components must vanish, we arrived to a
contradiction, showing that an \adapted Heegaard diagram does not
contain domains encountered in Case {\bf {(III.1)}}.

In case~ {\bf {(III.2)}} the two points with multipicity $\frac{1}{2}$
must coincide, hence in particular they are on the same boundary
component.  Abstractly, the surface $S$ from Theorem~\ref{thm:lip3}
is an annulus with four marked points on one component, and two on the
other.  For the component with two marked points, one of the arcs is
mapped onto an entire $\alpha$ (or $\beta$) circle, while the other
arc maps degree zero onto the corresponding $\beta$ (or $\alpha$)
circle. Abstractly such a domain is illustrated in
Figure~\ref{f:RectangleMinusDisk}; this is Case~\ref{case:Annulus} in
the proposition. Such domains can be found in adapted Heegaard
diagrams, as indicated in Figure~\ref{f:RectangleMinusDiskRealized}.

\begin{figure}[ht]
\begin{center}
\setlength{\unitlength}{1mm}
{\includegraphics[height=3cm]{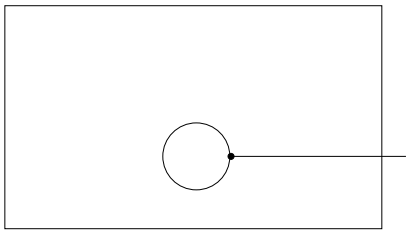}}
\end{center}
\caption{Schematic picture of domain in  Case~{\bf {(III.2)}}.} 
\label{f:RectangleMinusDisk}
\end{figure}

\begin{figure}[ht]
\begin{center}
\setlength{\unitlength}{1mm}
{\includegraphics[height=3cm]{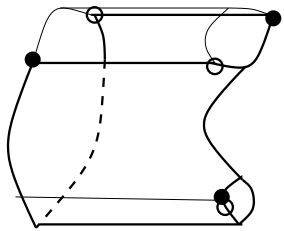}}
\end{center}
\caption{The realization of the domain of case {\bf {(III.2)}} in an
  adapted Heegaard diagram.}
\label{f:RectangleMinusDiskRealized}
\end{figure}

\noindent {\bf Case IV:} $d=4$

\noindent When $d=4$, there is only one possibility for the local
multiplicities to sum up to 2: we must have eight corner points, each
of local multiplicity $\frac{1}{4}$.  Once again, the count of Euler
characteristics and boundary components shows that there are two cases
to consider: either the domain has three boundary components, or it
has one. In the first case two boundary components consist of a pair
of $\alpha$-- and $\beta$--circles, intersecting each other twice,
hence admitting the property that $\pi $ is not injective on them. The
application of Lemma~\ref{lem:TwoNewCurves} (and the assumption
$n_w(\phi )\leq 1$) then shows the nonexistence of this
possibility. Therefore we are left with the case of a single boundary
component and a surface of Euler characteristic 1, i.e., an octagon
with local multiplicity $\frac{1}{4}$ at each of its eight corner
points, giving Case~\ref{case:Octagon}, and concluding the proof of
Proposition~\ref{prop:OnlyIndexOnes}. Notice that by
Equation~\eqref{e:branch} and Remark~\ref{r:sharp} in this case the
map $u$ given by Theorem~\ref{thm:lip3} is a local diffeomorphism.
\end{proof}

\section{Embeddedness of  Maslov index 1 domains} 
\label{sec:CalcContributions}

Suppose that $\phi\in \pi_2(\x ,\y )$ is a homology class of Whitney
disks with Maslov index 1 and $n_w(\phi)\leq 1$. According to
Proposition~\ref{prop:OnlyIndexOnes}, $\phi$ is represented by a
domain which is the image (in the sense of Theorem~\ref{thm:lip3}) of
a surface which is either a rectangle, an annulus of the shape given
by Figure~\ref{f:RectangleMinusDisk}, or an octagon.  The aim of the
present section is to strenghten this result, by showing that in each
of those cases, the domain is embedded (in the sense of
Definition~\ref{def:Domain}). The embeddedness of the domains will
then allow us to compute $c(\phi )$ and conclude the proof of
Theorem~\ref{t:contribution}.

\begin{thm}
  \label{thm:ClassifyDomains}
  Let $\phi\in\pi_2(\x,\y)$ be a non-negative homology class
  of Whitney disks in an adapted Heegaard diagram.
  Suppose moreover that $\phi$ has Maslov index one. 
  If  $n_w(\phi)=0$, then 
  \begin{itemize}
  \item $\phi$ is an embedded rectangle.
  \end{itemize}
  If $n_w(\phi)=1$, then once again $\phi$ is embedded. In fact,
  the domain of $\phi$ is one of the following two types
  of domain:
  \begin{itemize}
    \item it is an annulus with four marked points $x_1,y_1,x_2,y_3$ on 
      one boundary component, each of which has local multiplicity
      $\frac{1}{4}$, and one marked point $x_3=y_3$ on the other, with
      local multiplicity $\frac{1}{2}$, or
    \item it is a disk with eight marked boundary points each
      with multiplicity $\frac{1}{4}$ (i.e. an octagon).
  \end{itemize}
\end{thm}

The three cases above, of ourse, correspond to the to the three
cases enumerated in Proposition~\ref{prop:OnlyIndexOnes}.
We subdivide the proof into the three cases.
\begin{prop}(cf. also \cite{SW})
  \label{prop:Rectangle}
  If $\phi$ is as in Proposition~\ref{prop:OnlyIndexOnes}, and its
  domain is covered by Case~{\ref{case:Rectangle}} (i.e. it is an
  immersed rectangle), then the domain is embedded.
\end{prop}
\begin{proof}
  Suppose that all the unmarked elementary domains in a multi-pointed
  Heegaard diagram are rectangles (or bigons), i.e., the Heegaard
  diagram is nice. Then, as in~\cite{SW}, all non-negative domains
  with Maslov index one and $n_w(\phi )=0$ are rectangles (or bigons).
  In an \adapted Heegaard diagram the argument is slightly simpler, in
  fact, we claim that for a rectangle $D$ even the projection $\pi
  (\partial D)$ of the boudary $\partial D$ into the grid diagram is
  embedded. Indeed, the argument used in proving cases (II.\ref{k4})
  and (II.\ref{k5}) above applies and shows that the four corner
  points must project to four different points.  Now if the projection
  $\pi$ is not injective on a side $a_1$, then (in order $\pi
  (\partial S)$ to represent zero in homology) the same should hold
  for the opposite side $a_2$ (since all the local multiplicities are
  equal to $\frac{1}{4}$), which by Lemma~\ref{lem:TwoNewCurves}
  implies that $n_w(\phi)\geq 2$, contradiction our assumption.  Thus,
  $\partial D\subset \Sigma$ is an embedded, null-homologous curve.
  As such, it divides $\Sigma$ into two components where the local
  multiplicity is constant. Since there is one corner with local
  multiplicity $\frac{1}{4}$, these two constants are $0$ and $1$:
  i.e. the domain of $\phi$ is embedded.
\end{proof}
\begin{prop}
  \label{prop:Annulus}
  If $\phi$ is as in Proposition~\ref{prop:OnlyIndexOnes}, and its
  domain falls under Case~\ref{case:Annulus} (i.e.  it is an immersed
  rectangle-minus-a-disk), then its domain is embedded.
\end{prop}

\begin{proof}
  Suppose that $\phi \in \pi_2(\x,\y )$ is a non-negative domain with
  Maslov index one, and $S$ is an annulus (provided by
  Theorem~\ref{thm:lip3}) as indicated in
  Figure~\ref{f:RectangleMinusDisk}.  Let $\gamma_1\subset S$ denote
  the circular boundary component of the surafe, while $\gamma_2
  \subset S$ the rectangular one.
  
  Our first goal is to show that image $c_1\cup c_2$ of $\gamma_1\cup
  \gamma_2$ under the map $u$ is an embedded one-manifold in
  $\Sigma$. Note that the curve $\gamma_1$ covers an $\alpha$- or
  $\beta$-circle; the two cases are similar, so we assume the former,
  and write $\alpha_i$ for the $\alpha$-circle covered by $\gamma_1$.
  Since $\alpha _i$ is a simple closed curve in $\Sigma$, and the
  local multiplicities in the corners on this boundary components are
  $\frac{1}{2}$, we get that $c_1=u(\gamma_1)$ is an embedded curve.
  We claim that $c_2=u(\gamma_2 )$ is also an embedded curve in the
  Heegaard surface $\Sigma$.  By the local multiplicities at the
  corner points it is obvious that $u$ is injective on the corners of
  $\gamma _2$. In fact, the same argument shows that $u$ is injective
  on each arc connecting two neighboring corner point.  Consider the
  projection $\pi(c_2)$.  The same analysis used earlier (for example,
  in excluding Case (II.\ref{k4}) in
  Proposition~\ref{prop:OnlyIndexOnes}) shows that the four corner
  points on $c_2$ map to four different points under $\pi $.  Since
  $c_1$ and $c_2$ are homologous to each other in $\Sigma$, we get
  that $\pi (c_1)$ is homologous to $\pi (c_2)$ in $T^2$, and $\pi
  (c_1)$ (as the image of an $\alpha$--curve) has no $B$--component.
  Now if the projection is not injective on a boundary $\beta$--arc in
  $c _2$, then it is not injective on the other boundary $\beta$--arc
  as well (since the $B$--component of the homology class of $T^2$
  determined by the projection of $c_2$ is 0), and so
  Lemma~\ref{lem:TwoNewCurves} applies again and gives a
  contradiction.  Therefore $\pi$ is injective on the $\beta$--sides
  of $c_2$, which now easily implies that $c_2$ is embedded in
  $\Sigma$.


  Thus, the only remaining possible singularities in $c_1\cup c_2$ are
  the intersection points between $c_1$ and $c_2$. The projections
  $\pi(c_1 )$ and $\pi(c_2 )$ can meet in at most two points $p_1$ and
  $p_2$, as it is illustrated in the second picture in
  Figure~\ref{f:poss}.  However, in this case, we claim that $c_1$ and
  $c_2$ are nonetheless disjoint: otherwise, there would be some
  region adjoining $p_i$ where the multiplicity of the domain $\phi$
  is $2$.  Considering the other corners of this region, we would be
  able to conclude also that the local multiplicity of $\phi$ at some
  region adjacent to one of the corners of $\phi$ is also $2$, a
  contradition to our hypothesis on $\phi$ (according to which all
  local multiplicities of $\phi$ near its corner points are $\leq 1$).
  
  Now, $c_1\cup c_2$ divides $\Sigma$ into two components (since $c
  _1$ is homologically non-trivial even when projected to $T^2$, and
  $c _2$ is homologous to $c _1$), thus there are exactly two possible
  local multiplicities for any interior point in the domain of
  $\phi$. Since near each corner point, both both $0$ and $1$ appear
  as local multiplicities, it follows that our domain $\phi$ is
  embedded.
\end{proof}

\begin{figure}[ht]
\begin{center}
\setlength{\unitlength}{1mm}
{\includegraphics[height=1cm]{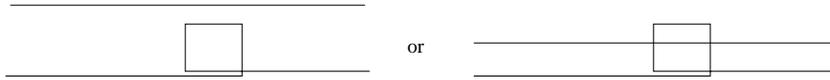}} 
\end{center}
\caption{Possibilities for the projection of $\partial D(\phi )$ for a 
homology class $\phi $ is Case~\ref{case:Annulus}.}
\label{f:poss}
\end{figure}

It remains to consider now Case~\ref{case:Octagon}, the case of
octagons.  The argument to verify this last case of
Theorem~\ref{thm:ClassifyDomains} involves a case-by-case
analysis. Before turning to this enumeration, we start with a simple
result providing a condition for the boundary of an octagon to be
embedded.

\begin{lem}\label{l:kizar}
Suppose that $\phi\in\pi_2(\x,\y)$ is a non-negative homology class of
Whitney disks with Maslov index one in an adapted Heegaard diagram,
and consider the corresponding map $u \colon S\to \Sigma $ from
Theorem~\ref{thm:lip3}. Suppose moreover that $u$ is not an embedding
along $\partial S$. Then there is a vertex $P$ of $S$ and a point
$x\in \partial S$ distinct from all the vertices such that $u (P)=u
(x)$. In particular, $\partial D(\phi )$ contains at least one new
$\alpha$-- and one new $\beta$--curve, on which the projection $\pi $
is not injective.
\end{lem}
\begin{proof}
To set up notation, let the octagon $S$ be as given by
Figure~\ref{f:octa}, with the sides of it denoted by $a _1 , \ldots ,
a _4$ and $ b_1, \ldots , b _4$.  Since by Equation~\eqref{e:branch}
and Remark~\ref{r:sharp} the map $u \colon S \to \Sigma$ in this case
is a local diffeomorphism, it equips $S$ with a tiling (as in the
proof of \cite[Theorem~3.2]{SW}, cf. also \cite{Lip}). In our case,
since we have $n_w(\phi)=1$, the tiling will contain a single octagon
and a number of rectangles. Let $A_i$ denote the number of
subintervals a side $a _i$ is cut into by the $\beta$--curves, and
define $B_j$ similarly.
\begin{figure}[ht]
\begin{center}
\setlength{\unitlength}{1mm}
{\includegraphics[height=6cm]{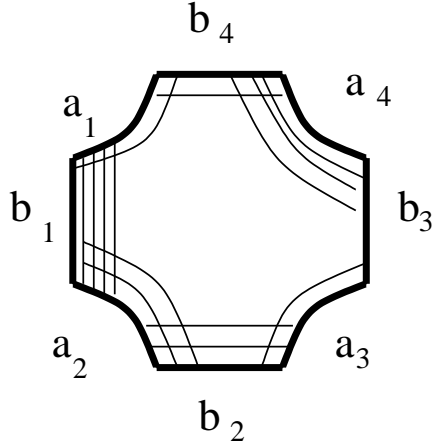}}
\end{center}
\caption{The octagon with the tiling.}
\label{f:octa}
\end{figure}

Suppose that $u$ is not an embedding on $\partial S$, i.e. there is a
point $x_0\in a_1$ which maps to the same point as some $y_0\in b _i$.
If $i=1$ (or $i=4$), then the lemma is proved: if $u(x_0)$ and
$u(y_0)$ are both intersections of $\alpha _1$ and $\beta _1$ then
(since these curves have exactly two intersections) one of them must
be a corner point and the other one cannot be a corner point. (A
similar argument applies for $i=4$.)  The arguments for $i=2$ or $i=3$
are symmetric, so we can assume that $i=2$. The point $y_0$ now can
move along the preimage via $u$ of an $\alpha$--curve, intersecting
$b_2$; on the tiling it either moves towards $b_1$ or $b_3$.  Every
step $y_0$ makes is mirrored by a step of $x_0$ on $a_1$. Since the
local multiplicity of any corner point is $\frac{1}{4}$, in this
procedure $y_0$ will reach $b_1$ or $b_4$ before $x_0$ moves to the
endpoint of $\alpha_1$. If $y_0$ reaches $b_1$, we get the desired
second intersection of the image of $a_1$ and $b _1$, and the
statement is proved.  If it moves towards $b_3$, then it will take
exactly $A_3$ steps before it reaches that boundary component, hence
we conclude that $A_1>A_3$.  The terminal points of this procedure are
denoted by $x_1$ and $y_1$ respectively.  Now using the same principle
for $x_1$ and $y_1$, by moving $x_1$ we either prove the statement (if
$x_1$ moves towards $a_4$), or we get $B_1<B_3$ and a new point pair
$x_2$ and $y_2$ (now on $a_2$ and $b_3$). Repeating the same
procedure, eventually we either get the desired configuration, or
derive the inequality $A_3>A_1$, which contradicts the inequality
$A_1>A_3$ found above.
\end{proof}
Now we are to discuss the last case in the proof of
Theorem~\ref{thm:ClassifyDomains}.

\begin{prop}
  \label{prop:Octagon}
  If $\phi$ is as in Proposition~\ref{prop:OnlyIndexOnes}, 
  and its domain falls under Case~\ref{case:Octagon} (i.e.
  it is an immersed octagon), then its domain
  is embedded.
\end{prop}

\begin{proof}
Suppose now that $P_1, \ldots , P_8$ denote the eight corners of the
octagon, and for a map $u\colon S \to \Sigma $ representing $\phi$,
consider the points $u (P_i)$.  Since the local multiplicity of each
is equal to $\frac{1}{4}$, these points are all distinct. Define $t_{u
}$ as the cardinality of the set $\{ \pi (u (P_i ))\}\subset
T^2$. Since any intersection point in the extended grid has at most
two intersection point preimages in $\Sigma$, we get that $t_{u }$ is
between 4 and 8. As in the proof of Proposition~\ref{prop:Annulus}, we
start by showing that $u$ is an embedding along $\partial S$. The
somewhat lenghty argument is divided into cases acording to the value
of $t_u$.

If $t_{u }=4$ then there is no point which could play the role of $x$
of Lemma~\ref{l:kizar}, since $x$ was not a corner, while the preimage
of each $u (P_i)$ consists of only corners in this case. Said it in
another way: the curves connecting the images must all be old curves
(since they contain four coordinates of $\x$ and $\y$), but according
to Lemma~\ref{l:kizar} the boundary $\partial S$ is embedded if all
the curves are old curves.

We exclude the possibility that $t_{u }=5$ as follows.  By
Property~\eqref{property:ProjectGenerator} of
Proposition~\ref{p:Properties}, for the two intersection points $\x$
and $\y$ any curve in $T^2$ contains an even number of projections of
coordinates. Now consider a point of the torus to which only one $u
(P_i)$ projects. Consider either the $\alpha$-- or the $\beta$--curve
passing through it, whichever does not contain the other point with a
single preimage of the type $u (P_i)$. (Since there is only one
further such point, there is a curve with this property.) Now this
curve contains a single coordinate, plus the ones coming from the
points having two preimages among the $u (P_i)$, changing the number
of coordinates by 2. Therefore this number will have the wrong parity,
concluding the analysis of the case $t_{u}=5$.

Assume next that $t_{u}=6$.  In this case there are 4 points on the
torus $T^2$ (which we will denote by $Q_1, Q_2, Q_3, Q_4$) with single
preimages, and two ($Q_5$ and $Q_6$) with two preimages. The argument
given above with the parities shows that $Q_1, Q_2, Q_3, Q_4$ must
determine a rectangle in the grid, determining four curves ${\mathfrak
  a}_1, {\mathfrak a}_2$ and ${\mathfrak b}_1, {\mathfrak b}_2$.  The
last 2 points $Q_5$ and $Q_6$ must be on these four curves. If both
$Q_5, Q_6$ are on ${\mathfrak a}$--curves (say, on ${\mathfrak a}_1$
and ${\mathfrak a}_2$) then they must share a ${\mathfrak b}$--curve,
otherwise the ${\mathfrak b}$--curves on which they are located must
be new curves in the grid diagram (since we need to pass from one
sheet to the other on the curve), and the existence of two new curve
segments in the boundary on which the projection is non-injective (by
Lemma~\ref{lem:TwoNewCurves}) contradicts $n_w(\phi )=1$.  Thus, the
configuration should look like the one given in
Figure~\ref{f:hat}(a). Now, however, the horizontal curves are old
curves (since they carry 4 coordinates), hence $\partial S$ must be
embedded since there is no horizontal new curve in its image.
\begin{figure}[ht]
\begin{center}
\setlength{\unitlength}{1mm}
{\includegraphics[height=6cm]{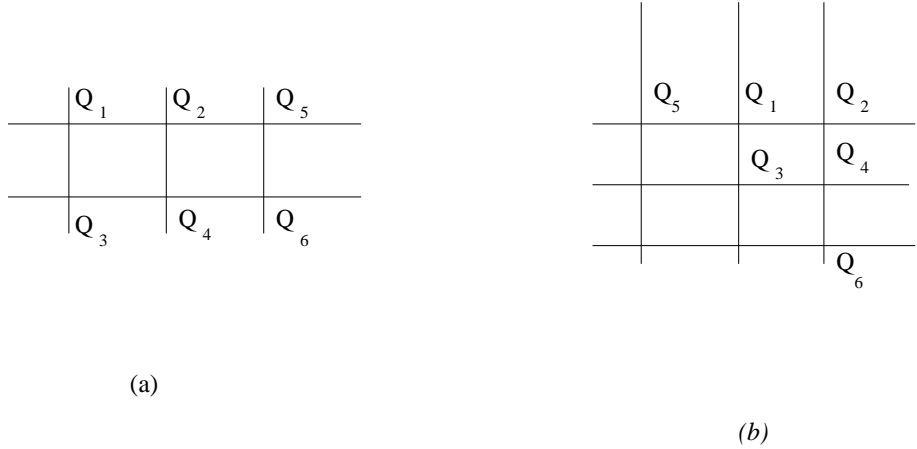}}
\end{center}
\caption{Two cases for $t_{u }=6$.}
\label{f:hat}
\end{figure}
There is a further case with $t_u=6$ we have to consider, namely when
$Q_5$ and $Q_6$ do not sit both on the ${\mathfrak {a}}$-- or
${\mathfrak {b}}$--curves determined by the first four $Q$'s. Such a
situation is shown by Figure~\ref{f:hat}(b). Obviously the vertical
curve passing through $Q_5$ (and the horizontal through $Q_6$) are new
curves, while both curves passing through $Q_2$ are old (containing 4
coordinates in the projection). In order the boundary $\partial S$ not
to be embedded, we need two new curves, which intersect each other in
a corner point and in a further boundary point. The only candidate in
our picture for such curves are the two curves passing through
$Q_3$. Since these must be parts of new curves with the property that
$\pi$ is not injective on them, considering them and the new curves
through $Q_5$ and $Q_6$, Lemma~\ref{lem:TwoNewCurves} applies and
delivers the required contradiction.  This argument concludes the
analysis of the case when $t_{u }=6$, and shows that $\partial S$ is
embedded in this case.

When $t_{u }=7$, there are three significantly different cases to
consider.  By $t_u=7$ there is one point with two corner point
preimages (symbolized by a circle on Figure~\ref{f:het}) and six with
a single corner point preimage (denoted by x's).  By cutting the torus
along an appropriate pair of lines we can assume that the circle is in
the upper left corner of the diagram. The circle must be connected to
at least one of the x's, but the evenness of multiplicities along a
line implies that there should be one more x on that line.  Without
loss of generality we can assume that these two x's are on a
horizontal line passing through the circle. The orthogonal lines
through these two x's must contain further x's, which are either in a
single horizontal line, or on two distinct ones. If these two
additional x's (x$_3$ and x$_4$) are on the same horizontal line, then
the last two x's cannot be put down without violating the evenness
condition. This shows that x$_3$ and x$_4$ are not on the same
horizontal line.  The last two x's must be on these two horizontal
lines, and (again by the evenness condition) they should be in the
same vertical one. This vertical line might be disjoint from the one
passing through the circle (cf. Figure~\ref{f:het}(a)), or is the one
passing through the circled point (as it is given by
Figures~\ref{f:het}(b) and (c)).  Notice that for
Figure~\ref{f:het}(a) there are furter combinatorial possibilities,
but our subsequent argument will be insensitive for te further
combinatorial distinction of these diagrams, hence we do not enumerate
them.

\begin{figure}[ht]
\begin{center}
\setlength{\unitlength}{1mm}
{\includegraphics[height=5cm]{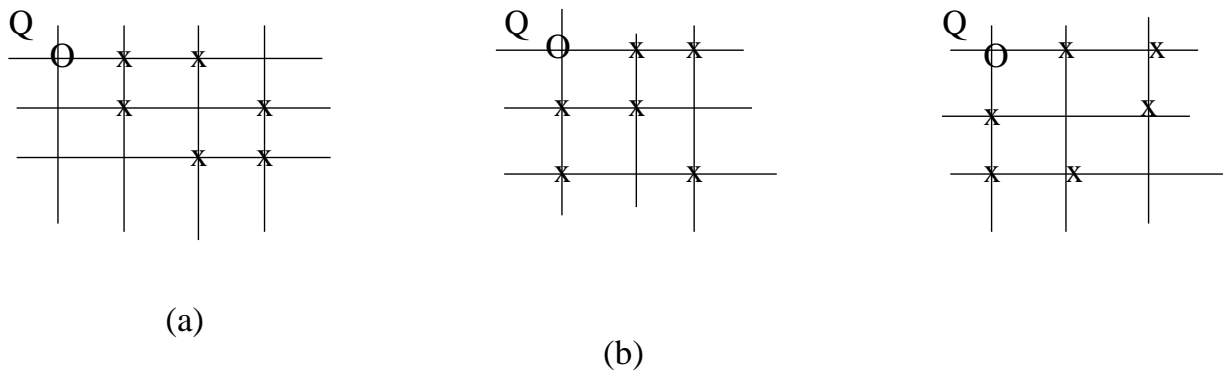}}
\end{center}
\caption{{\bf Cases for $t_{u }=7$.}
Case (a) has six versions, by permuting the three vertical lines
which do not contain $O$, but the various possibilities can be excluded by the
same principle.}
\label{f:het}
\end{figure}
In Figure~\ref{f:het}(a) the vertical line passing through the circle
is obviously a new curve and $\pi $ is not injective on it. By
Lemma~\ref{l:kizar} the boundary $\partial S$ is embedded unless there
is a horizontal and a vertical new curve which passes through one
corner point and one other boundary point which is not a corner.  If
such a vertical new curve exists, it must be different from the one
passing through the circled point, hence we get two parallel new
curves, each having the property that the projection $\pi $ is not
injective on it, which (by Lemma~\ref{lem:TwoNewCurves}) leads to a
contradiction with the assumption $n_w(\phi)=1$.

For the cases of Figures~\ref{f:het}(b) and (c) we argue as follows:
both curves passing through the circled point $Q$ must be old (since
the sum of the coordinates is equal to four). By traversing along the
curves as $\pi (\partial (D(\phi )))$ does, the segments on these
curves (two on each) are oriented in the same direction, and cannot
perform complete turns because the local multiplicities are equal to
$\frac{1}{4}$. The orientation of the image of $\partial (D(\phi ))$
on both horizontal lines not passing through $Q$ are opposite to the
orientation we get on the line passing through $Q$. To verify these
statements we have to cosnider all possible cases for the image of
$\partial D(\phi )$. The combinatorial difference between
Figures~\ref{f:het}(b) and (c) {\em a priori} could mean that we get
differente results in this oriantation issue, but the careful analysis
of all cases shows that for both cases the statement above holds
true. This implies that, in order to get a null--homologous curve,
none of these curves can go around the torus (since they cannot cancel
each other by the orientation reasoning). So even if one of the curves
not passing through $Q$ is a new curve, the projection $\pi$ remains
injective on it, so by Lemma~\ref{l:kizar} the boundary $\partial S$
is embedded.  This argument finishes the analysis when $t_u=7$.

Assume finally that $t_{u }=8$. Consider a curve passing through a
chosen point $\pi (u (P_i))$. It contains an odd number of further
such points, and obviously cannot contain more than three additional
ones.  Therefore we have to distinguish two cases:

\noindent (1) there is a line containing 4 points, or

\noindent (2) each line contains exactly 2.

In the first case, the perpendicular lines must contain an odd number
of further points, hence that odd number must be one.  The possibilities are
pictured in Figure~\ref{f:nyolcpont}, since these additional four
points then might lie on one or two lines, giving (i) and (iii).  In
case (i) both horizontal lines are old curves (having four
coordinates), hence there is no new horizontal curve in the diagram,
which according to Lemma~\ref{l:kizar} implies that $\partial S$ is
embedded.

\begin{figure}[ht]
\begin{center}
\setlength{\unitlength}{1mm}
{\includegraphics[height=7cm]{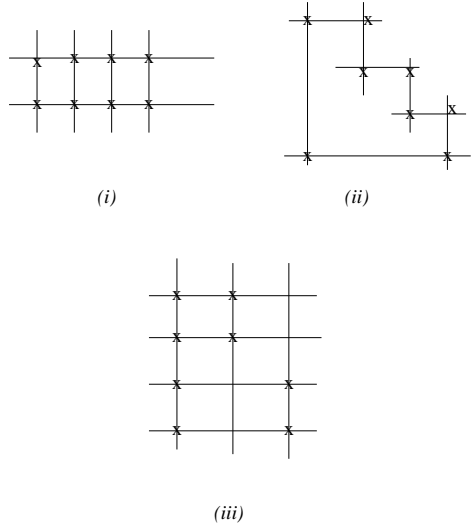}}
\end{center}
\caption{Three types of  cases for $t_{u }=8$.}
\label{f:nyolcpont}
\end{figure}

In order to analyze possibilities regarding cases (ii) and (iii), we
first prove a lemma.  To state this result, suppose that the points
$\pi (u (P_i))$ are given on the extended grid. In the spirit of
Lemma~\ref{lem:Principle}, we specify how two neighbouring points can
be connected with a short interval (for which we have four
possibilities: there are two short intervals on the torus connecting
two points, and each can be oriented in two ways).  The initial choice
creates a unique path, since when connecing all points we have to
follow the convention that at every point we turn $90^{\circ}$ to the
left (dictated by the fact that the local multiplicity is
$\frac{1}{4}$ at every point) and then use the short interval
connecting the point to the next one.  In this manner we get a closed
curve $\gamma$ in the extended grid, which might not be the projection
of the image of $\partial S$ in the Heegaard diagram, but it visits
the corner points in the same order as the projection of $u (\partial
S)$. As always, $A$ and $B$ will denote the homology generators of
$H_1(T^2; \Z )$. With these conventions in place now we are ready to
state
\begin{lem}\label{l:megkizar}
Fix the images of the corner points of the octagon $S$.  If for all
four choices the homology class represented by the corresponding
$\gamma$ is different from $\pm A \pm B$ then any map $u \colon S \to
\Sigma$ of an octagon representing $\phi \in \pi _2 (\x , \y )$ with
$\Mas(\phi)=1$ and with the given configuration of points $\pi (\phi
(P_i))$ is an embedding on $\partial S$.
\end{lem}
\begin{proof} 
  Suppose that $\partial S$ is not embedded, which by
  Lemma~\ref{l:kizar} implies, in particular, that there are new
  $\alpha$-- and $\beta$--curves such that their projections are not
  embedded. In order to recapture the image of $\partial S$ from
  $\gamma$, therefore we must modify its homology class.  Since we
  cannot alter the order and the turns, there is only one operation we
  can perform on $\gamma$: we can add to the short interval connecting
  two consecutive corners a complete turn, cf. the similar argument in
  Lemma~\ref{lem:Principle}.  Notice that we cannot do it along an old
  curve (since this would increase the local multiplicity of the
  coordinates along that curve), and we cannot add two such twists
  along parallel new curves (cf. Lemma~\ref{lem:TwoNewCurves}) or two
  turnes along the same curve.  Thus, if $\partial S$ is not
  embedded by $u$, we must change $\gamma$ by $\pm A \pm B$ to get
  $\partial S$.  Since the image of $\partial S$ is nullhomologous,
  the claim follows.
\end{proof}

To apply this result, we need to list all possible projections of the
8 corner points, and compute the corresponding $\gamma$'s. 
This is a fairly straightforward exercise. For 
case (iii) of Figure~\ref{f:nyolcpont} there are two
significantly different cases, all others can be derived by either
switching the roles of the $\alpha$-- and $\beta$--curves or cutting
the torus open to a rectangle along different lines. These two cases are
shown by Figure~\ref{f:twoc}.
\begin{figure}[ht]
\begin{center}
\setlength{\unitlength}{1mm}
{\includegraphics[height=4cm]{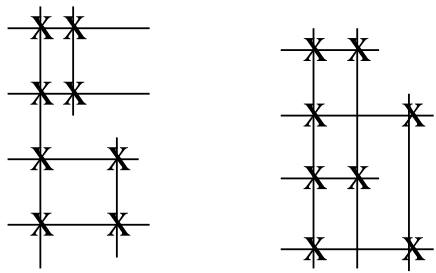}}
\end{center}
\caption{Two cases to consider for (iii).}
\label{f:twoc}
\end{figure}
Each diagram gives rise to eight $\gamma$--curves. Specifically, the
vertical segment of the $\gamma$ curve through the upper-left x in
each picture can be connected to one of two x's in the same vertical
circle, with two possible orientations and in two possible ways.  A
simple case-by-case check shows that none of them is of the form $\pm
A \pm B$; in fact, in all cases the homology of $\gamma$ contains at
most one of the factors $\pm A$ or $\pm B$.

The discussion of case (ii) of Figure~\ref{f:nyolcpont} needs more
cases to consider. First we list all possible configurations of the
eight corner points in a $4\times 4$ grid. It is easy to see that
(when viewed on the torus) there is always a pair of points which are
neighbours: if this does not hold, the boundary would fall into two 
components. We can assume that these two points are on the top left
corner. By considering rotational and reflection symmetries
(which obviously would provide similar results) there are four cases to 
consider, and these configurations are shown by Figure~\ref{f:fourposs}.
\begin{figure}[ht]
\begin{center}
\setlength{\unitlength}{1mm}
{\includegraphics[height=6cm]{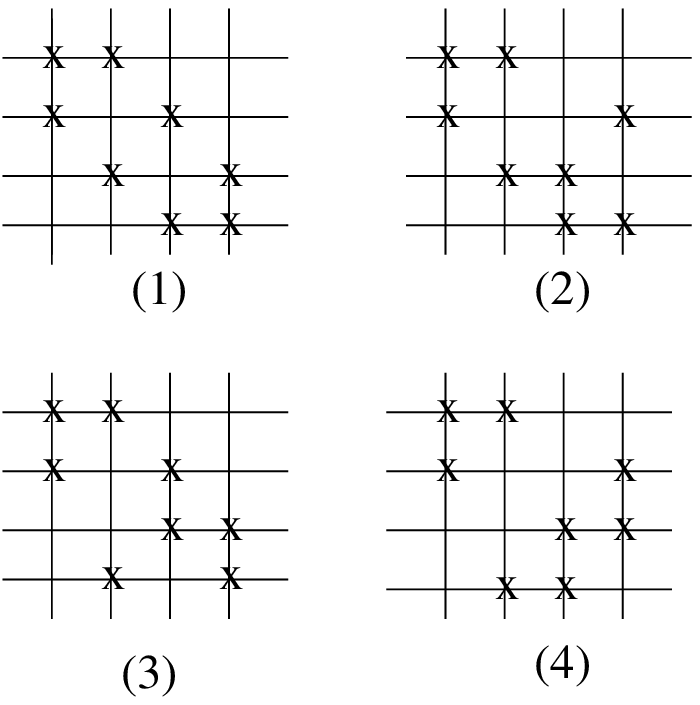}}
\end{center}
\caption{The four possibilities.}
\label{f:fourposs}
\end{figure}
Remember that each configuration gives rise to four $\gamma$--curves,
and it is not hard to check that for (1) the results are all
null--homologous, and for (2) and (3) they represent $\pm A$ (or $\pm
B$, depending on the chosen conventions). By Lemma~\ref{l:megkizar}
therefore in these cases the boundary $\partial S$ is embedded. The
configuration Figure~\ref{f:fourposs}(4), however, provides a curve
$\gamma$ with homology equal to $A+B$ (after appropriate orientations
are fixed), hence the lemma itself is not enough to verify that
$\partial S$ is embedded. The four possible $\gamma$'s are (up to
symmetries) equivalent to the diagram shown in Figure~\ref{f:gamma}.
\begin{figure}[ht]
\begin{center}
\setlength{\unitlength}{1mm}
{\includegraphics[height=5cm]{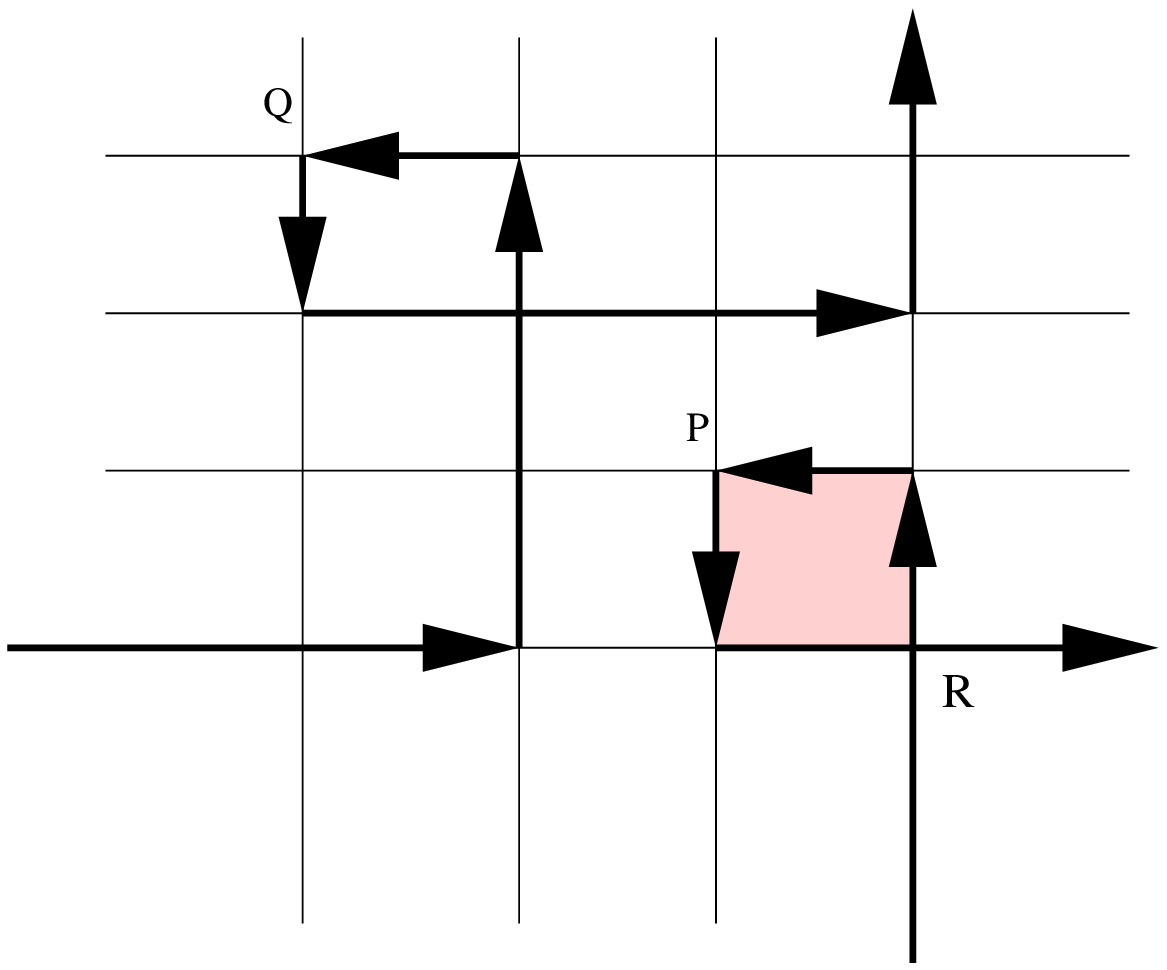}}
\end{center}
\caption{The curve $\gamma$ when Lemma~\ref{l:megkizar} does not apply.}
\label{f:gamma}
\end{figure}
From $\gamma$ we can actually recover the projection of $\partial S$:
since it bounds $S$ it must be homologically trivial, turns as
$\gamma$ does, and cannot pass an old curve more than once (since the
local multiplicities of the corners are equal to $\frac{1}{4}$). To
turn $\gamma$ to a trivial element in the homology of the torus, we
have to add a full turn both to the horizontal and to the vertical
direction either at $P$ or at $Q$; these possibilities are dictated by
the orientation (as shown by Figure~\ref{f:gamma}). If we add the
horizontal and the vertical full turns in different points (one at $P$
and the other one at $Q$), then by Lemma~\ref{l:kizar} the resulting
image of $\partial S$ comes from an embedded curve. If the two full
twists are added at the same point, we need to argue further.  The two
choices as where to add the two full turns will provide similar
results, so it is enough to examine the case when we do the
modification at $Q$. Since these curves pass more than a complete turn
down in the grid, these must be new curves (turning at least half),
hence by Lemma~\ref{lem:ThereIsABasepoint} there are base points near
them.  Since we have a single base point in the domain, it must be in
the square with vertex $Q$. Consider now the inverse images of the
arcs passing through the point $R$.  These arcs are disjoint in the
Heegaard diagram for $Y$ only if there is an octagon which projects
onto the shaded region, contradicting the fact that the base point is
in the elementary domain which projects to the rectangle adjacent to
$Q$.  On the other hand, if the arcs intersect, then a simple
calculation of multiplicities shows that the domain to the lower right
from $R$ has multiplicity $-1$ (since the domain mapping to the shaded
one in the grid has multiplicity 1), providing a contradiction to the
fact that $d(\phi )$ is non-negative.

In conclusion, by examining all possible values of $t_u$ we verified
that the boundary of the surface $S$ is embedded. Since it is a
homologically trivial simple closed curve, its complement has two
components, therefore in writing $D(\phi )$ as the formal linear
combination of elementary domains we have only at most two distinct
coefficients. Since the corner points have local multiplicities
$\frac{1}{4}$, the domains around each corner point have
multiplicities $0$ or $1$, hence the domain $D(\phi )$ is embedded (in
the sense of Definition~\ref{def:Domain}).
\end{proof}

\subsection*{The proof of the main theorem}

After this case-by-case analysis we are ready to prove the main theorem
of the paper.

\begin{proof}[Proof of Theorem~\ref{thm:ClassifyDomains}]
  The classification of non-negative, Maslov index one domains with
  $n_w(\phi)\leq 1$ as stated in Theorem~\ref{thm:ClassifyDomains} is
  a consequence of Proposition~\ref{prop:OnlyIndexOnes}, together with
  its strengthenings, Propositions~\ref{prop:Rectangle},
  \ref{prop:Annulus}, and \ref{prop:Octagon}.
\end{proof}

\begin{proof}[Proof of Theorem~\ref{t:contribution}]
  According to Theorem~\ref{thm:ClassifyDomains}, if $\phi$ is a
  non-negative, Maslov index one domain with $n_w(\phi)\leq 1$, then
  it is one of three special types of domains; moreover, we know that
  $c(\phi)$ does not depend on a choice of almost-complex structure.
  We can now show for all three types of domains that $c(\phi)=1$
  holds, for example, as follows.  We embed $\phi$ into a Heegaard
  diagram for $S^3$ for which there are three generators, in which
  $\phi$ is the non-negative domain which does not contain the
  basepoint. Since $\HFa (S^3)=\Z / 2\Z$, this Heegaard diagram shows
  that $c(\phi )=1$ mod 2. (For the case of the octagon, this is
  explictly done in \cite[Theorem~6.1]{KT}; and a more general
  construction is given in \cite[Lemma~3.11]{Cube} which handles all
  three cases.) Note that for the case of the annulus, the
  combinatorics of the diagram is important, in order for the answer
  to be independent of the choice of almost-complex structure.  Since
  $\phi $ is embedded in the adapted Heegaard diagram for $Y$, its
  contribution $c(\phi )$ in the differential is the same as in the
  Heegaard diagram for $S^3$, concluding the proof.
\end{proof}

\begin{proof}[Proof of Theorem~\ref{t:main}]
  According to Corollary~\ref{c:van}, every three-manifold admits an
  adapted Heegaard diagram. According to Lemma~\ref{lem:CalculateCF2},
  $\HFtwo(Y)$ can be determined once one calculates $c(\phi)$ for
  every Maslov index one domain with $n_{w}(\phi)\leq 1$. According to
  Theorem~\ref{t:contribution}, however, for such $\phi$, we have that
  $c(\phi)=1$ if and only if $\phi$ is non-negative.
\end{proof}

\bibliographystyle{plain}
\bibliography{biblio}

%

\end{document}